\documentclass[11pt]{amsart}
\usepackage[T1]{fontenc}
\usepackage{lmodern}
\usepackage{amssymb,mathtools,booktabs,array}
\usepackage[margin=1.12in]{geometry}
\usepackage[expansion=false]{microtype}
\usepackage{xurl}
\usepackage{xcolor}
\usepackage[colorlinks=true,linkcolor=blue!45!black,citecolor=blue!45!black,urlcolor=blue!45!black]{hyperref}

\newtheorem{theorem}{Theorem}[section]
\newtheorem{lemma}[theorem]{Lemma}
\newtheorem{proposition}[theorem]{Proposition}
\newtheorem{corollary}[theorem]{Corollary}
\theoremstyle{remark}

\numberwithin{equation}{section}
\DeclareMathOperator{\sinc}{sinc}
\DeclareMathOperator{\exprel}{exprel}
\DeclareMathOperator{\Si}{Si}
\DeclareMathOperator{\Ci}{Ci}
\DeclareMathOperator{\Arg}{Arg}
\newcommand{\D}{\mathbb D}
\newcommand{\R}{\mathbb R}
\newcommand{\E}{\mathbb E}
\newcommand{\ii}{\mathrm i}
\newcommand{\dd}{\,\mathrm d}
\newcommand{\ga}{\Gamma}
\newcommand{\eps}{\varepsilon}
\newcommand{\file}[1]{{\small\path{#1}}}
\hypersetup{
 pdftitle={A proof of the Koumandos--Ruscheweyh conjecture},
 pdfauthor={Yicen Ma},
 pdfsubject={Positive trigonometric sums and sharp sector subordination},
 pdfkeywords={Koumandos--Ruscheweyh conjecture, computer-assisted proof, subordination}
}
\AddToHook{env/thebibliography/before}{\clearpage}
\allowdisplaybreaks[2]
\title[A proof of the Koumandos--Ruscheweyh conjecture]{A proof of the Koumandos--Ruscheweyh conjecture}
\author{Yicen Ma}
\address{College of Chemistry, Zhejiang University, Hangzhou, China}
\date{October 5, 2026 (Version 2)}

\begin{document}
\begin{abstract}
We prove the Koumandos--Ruscheweyh conjecture for every
$0<\rho\leq1$. If $\nu(\rho)$ is the unique root in $(0,1]$ of
$\int_0^{(1+\rho)\pi}t^{\mu-1}\sin(t-\rho\pi)\dd t=0$, then
$(1-z)^\rho s_n^\mu(z)\prec((1+z)/(1-z))^\rho$ for every $n\geq0$,
$0<\mu\leq\nu(\rho)$ and $z\in\mathbb D$, where
$s_n^\mu(z)=\sum_{k=0}^n(\mu)_kz^k/k!$.
The parameter $\nu(\rho)$ is optimal. The proof combines an
all-parameter gamma-coefficient comparison with an exact beta integral
and a binomial variance estimate, reducing the infinitely many degrees
to finitely many continuous interval inequalities. The remaining
inequalities are certified by 256-bit ball arithmetic; rational covers,
source code and alternative-formula verifiers are provided.
The weak positive-real-part conjecture follows as a sharp corollary.
We also derive sharp consequences for starlike functions and
Gegenbauer polynomial sections: a full-parameter convolution
subordination, the optimal starlike order for uniform partial-sum
sectors, and the optimal Gegenbauer parameter and sector angle.
The necessary bounds and angular sharpness are obtained from explicit
kernels and interior scaling limits.
\end{abstract}
\maketitle

\section{Introduction and main result}

For $\mu>0$ put
\[
 a_k(\mu)=\frac{(\mu)_k}{k!},\qquad
 s_n^\mu(z)=\sum_{k=0}^n a_k(\mu)z^k,\qquad
 \D=\{z\in\mathbb C:|z|<1\}.
\]
Here $(\mu)_0=1$ and $(\mu)_k=\mu(\mu+1)\cdots(\mu+k-1)$.
For $0<\rho\leq1$ define
\begin{equation}\label{eq:I}
 I(\mu,\rho)=\int_0^{(1+\rho)\pi}
                 t^{\mu-1}\sin(t-\rho\pi)\dd t,
 \qquad \nu(\rho)=\mu^*(\rho),
\end{equation}
where $\nu(\rho)$ is the unique zero of $I(\cdot,\rho)$ in $(0,1]$.
Existence and uniqueness will also be proved below.
All powers of $1-z$ are taken in the analytic branch equal to $1$ at
$z=0$.

Koumandos and Ruscheweyh \cite{KR} conjectured that $\nu(\rho)$ is
the largest parameter for which
\begin{equation}\label{eq:strong}
 (1-z)^\rho s_n^\mu(z)
 \prec \left(\frac{1+z}{1-z}\right)^\rho
\end{equation}
holds for every degree and every smaller positive $\mu$.
For analytic functions, $f\prec g$ means $f=g\circ\omega$ for an analytic
map $\omega:\D\to\D$ with $\omega(0)=0$.
The right-hand side of \eqref{eq:strong} maps $\D$ conformally onto
the sector $|\arg w|<\rho\pi/2$.

The case $\rho=1/2$ was established in \cite{KR}, and the case
$\rho=1/4$ in \cite{KL}. Sangal and Swaminathan \cite{SS}
established the strong conjecture in a neighborhood of $\rho=1/3$.
These works use positivity of trigonometric sums and comparisons
between gamma coefficients and oscillatory integrals.
The present proof treats the entire interval of parameters.

\begin{theorem}\label{thm:main}
Let $0<\rho\leq1$. For every integer $n\geq0$, every
$0<\mu\leq\nu(\rho)$, and every $z\in\D$, \eqref{eq:strong} holds.
Equivalently,
\[
 \left|\arg\bigl((1-z)^\rho s_n^\mu(z)\bigr)\right|
       <\frac{\rho\pi}{2}.
\]
For each fixed $\rho$, every $\mu>\nu(\rho)$ violates
\eqref{eq:strong} for some $n$ and some $z\in\D$.
\end{theorem}

\subsection{Immediate consequences of the main theorem}

\begin{corollary}[Weak positive-real-part conjecture]\label{cor:weak}
For $0<\rho\leq1$, $0<\mu\leq\nu(\rho)$, $n\geq0$ and $z\in\D$,
\begin{equation}\label{eq:weak}
 \Re\bigl((1-z)^{2\rho-1}s_n^\mu(z)\bigr)>0.
\end{equation}
For each fixed $\rho$, $\nu(\rho)$ is also the largest parameter
for which \eqref{eq:weak} holds for every degree and every smaller
positive $\mu$.
\end{corollary}

\begin{corollary}[The harmonic-mean assertion of Sangal--Swaminathan]
\label{cor:harmonic}
Let $0<\rho_1,\rho_2\leq1$ and
$H=2\rho_1\rho_2/(\rho_1+\rho_2)$. For every $n\geq0$ and
$0<\mu\leq\nu(H)$,
\[
 (1-z)^H s_n^\mu(z)\prec
 \left(\frac{1+z}{1-z}\right)^H,
 \qquad
 \Re\bigl((1-z)^{2H-1}s_n^\mu(z)\bigr)>0
 \quad(z\in\D).
\]
\end{corollary}
\begin{proof}
Since $0<H\leq1$, the assertion follows immediately from the full
parameter theorem, Theorem~\ref{thm:main}, and
Corollary~\ref{cor:weak}, applied at $\rho=H$.
\end{proof}
This establishes the assertion in \cite[Conjecture~3]{SS} without
needing its hypotheses at $\rho_1$ and $\rho_2$ separately.
That harmonic-mean conjecture is distinct from the starlike
Conjecture~3 of Koumandos--Ruscheweyh referred to in
\cite[after Theorem~4]{SS}.

Beyond the original conjecture, the full parameter theorem has
several sharp consequences. Section~\ref{sec:starlike-convolution}
gives the starlike convolution subordination for all admissible
parameters, including both parameter branches.
Section~\ref{sec:optimal-starlike} determines the optimal starlike
order and the smallest uniform half-angle for partial sums divided
by $z$. Section~\ref{sec:gegenbauer} gives the sharp uniform
Gegenbauer parameter and the exact enlarged range at $x=0$.
These are applications of the main theorem and classical convolution
theory; previously known special cases are identified there.
No additional finite numerical verification is needed for these
consequences.

\subsection{Structure of the proof}

The proof is computer assisted. Its infinite-degree reduction is
analytic; the finite inequalities are evaluated on complete parameter
and angle intervals with outward error bounds. In particular, the
argument does not extrapolate a finite degree scan.
The division of the critical-parameter proof is as follows:
\[
\begin{array}{c|c|c}
\text{parameter range}&\text{analytic large-degree threshold}
                     &\text{remaining degrees}\\ \hline
(0,2/5]&64&1,\ldots,63\\
{[2/5,999/1000]}&8&1,\ldots,7\\
{[999/1000,1)}&8&1,\ldots,7 .
\end{array}
\]
The last range is handled after extracting the vanishing parameter
$1-\nu$. The case $\rho=1$, the degree $0$, and the passage to
$\mu<\nu$ are elementary.

\section{Boundary reduction and the critical root}

Throughout the proof with $0<\rho<1$ we use the notation
\begin{equation}\label{eq:notation}
\begin{gathered}
 a=(1+\rho)\pi,\quad \nu=\nu(\rho),\quad p=1-\nu,\quad
 q=\frac{\nu}{\rho},\quad \sigma=\nu-\rho,\\
 \delta=\rho\pi-\frac{\nu\pi}{2},\qquad
 u=\sinc\phi=\frac{\sin\phi}{\phi},\qquad
 \ga_1=\ga(1+\nu).
\end{gathered}
\end{equation}
We set $\sinc0=1$.

\subsection{The two boundary sums}

Define
\begin{equation}\label{eq:T}
 T_n(\rho,\mu,\theta)
   =\sum_{k=0}^n a_k(\mu)\sin((2k+\rho)\theta),
                 \qquad 0\leq\theta\leq\pi.
\end{equation}
At $\mu=\nu$ write, for $0<\phi\leq\pi/2$,
\begin{align}
 L_n(\phi)&=T_n(\rho,\nu,\phi),\label{eq:L}\\
 R_n(\phi)&=T_n(\rho,\nu,\pi-\phi)
     =\sum_{k=0}^n a_k(\nu)
               \sin\bigl(\rho\pi-(2k+\rho)\phi\bigr).\label{eq:R}
\end{align}
Both sums must be considered: $T_n$ need not be symmetric.

\begin{lemma}\label{lem:boundary}
If $T_j(\rho,\nu,\theta)\geq0$ for all $j\geq0$ and
$0\leq\theta\leq\pi$, then \eqref{eq:strong} holds for all
$0<\mu\leq\nu$.
\end{lemma}
\begin{proof}
Let $F(z)=(1-z)^\rho s_n^\mu(z)$. For $z=e^{2\ii\theta}$,
$0<\theta<\pi$, direct calculation gives
\begin{align}
 \Re\bigl(e^{-\ii(1-\rho)\pi/2}F(z)\bigr)
  &=(2\sin\theta)^\rho T_n(\rho,\mu,\theta),\label{eq:rot1}\\
 \Re\bigl(e^{\ii(1-\rho)\pi/2}F(z)\bigr)
  &=(2\sin\theta)^\rho T_n(\rho,\mu,\pi-\theta).\label{eq:rot2}
\end{align}
Set $r_k=(\mu)_k/(\nu)_k$. The sequence $r_k$ is positive and
nonincreasing, because
$r_{k+1}/r_k=(k+\mu)/(k+\nu)\leq1$.
Abel summation gives
\begin{equation}\label{eq:abelmu}
 T_n(\rho,\mu,\theta)=r_nT_n(\rho,\nu,\theta)
       +\sum_{j=0}^{n-1}(r_j-r_{j+1})T_j(\rho,\nu,\theta).
\end{equation}
Thus the two boundary real parts are nonnegative.
The function $F$ extends continuously to the closed disk, with
$F(1)=0$. At $z=0$, both rotated real parts equal
$\sin(\rho\pi/2)>0$. The strong minimum principle for harmonic
functions makes both real parts strictly positive in $\D$.
Their two half-planes intersect in the open sector
$|\arg w|<\rho\pi/2$. Since $F(0)=1$, the conformal sector map
in \eqref{eq:strong} yields the required subordination.
\end{proof}

\subsection{Existence, bounds and monotonicity of the root}

\begin{lemma}\label{lem:tail}
If $0<v<1$ and $b>0$, the improper Fourier tail satisfies
\begin{equation}\label{eq:tail}
 \int_b^\infty t^{v-1}e^{\ii t}\dd t
   =\frac{e^{\ii b}}{\ga(1-v)}
       \int_0^\infty\frac{e^{-bs}s^{-v}}{s-\ii}\dd s,
 \qquad
 \left|\int_b^\infty t^{v-1}e^{\ii t}\dd t\right|<b^{v-1}.
\end{equation}
Also
$\int_0^\infty t^{v-1}e^{\ii t}\dd t
   =\ga(v)e^{\ii v\pi/2}$.
\end{lemma}
\begin{proof}
Insert the Laplace representation of $t^{v-1}$ and first add
an exponential damping factor $e^{-\eta t}$, $\eta>0$.
Fubini's theorem then gives the displayed tail formula with
$s-\ii$ replaced by $s+\eta-\ii$, and a factor $e^{-\eta b}$.
The resulting $s$-integrand is bounded by $e^{-bs}s^{-v}$,
which is integrable. Letting $\eta\downarrow0$ gives
\eqref{eq:tail}; the undamped $t$-integral converges by
Dirichlet's test. Since $|s-\ii|^{-1}<1$ for $s>0$, the strict
tail bound follows. The complete integral follows by taking
the limit in $\ga(v)(\eta-\ii)^{-v}$.
\end{proof}

\begin{lemma}\label{lem:root}
The root in \eqref{eq:I} exists uniquely for $0<\rho\leq1$.
For $0<\rho<1$ it satisfies
\begin{equation}\label{eq:rootbounds}
 \max\{\rho,2\rho-1\}<\nu(\rho)<\min\{1,2\rho\}.
\end{equation}
The function $\nu$ is differentiable and strictly increasing on
$(0,1)$. At the root,
\begin{equation}\label{eq:D}
 D:=\int_0^a t^{-p}\cos(t-\rho\pi)\dd t
       =\ga(\nu)\cos\delta+pK(a,p)>0,
\end{equation}
where
\begin{equation}\label{eq:K}
 K(a,p)=\frac1{\ga(p+1)}
       \int_0^\infty\frac{e^{-as}s^p}{1+s^2}\dd s,
       \qquad 0<K(a,p)<a^{-p-1}.
\end{equation}
\end{lemma}
\begin{proof}
For $\rho<1$, $I(\mu,\rho)\to-\infty$ as $\mu\downarrow0$,
while $I(1,\rho)=1+\cos(\rho\pi)>0$.
At any root, subtracting $\log(\rho\pi)I=0$ gives
\begin{equation}\label{eq:Imu}
 I_\mu=\int_0^a t^{\mu-1}
       \log\!\left(\frac{t}{\rho\pi}\right)
       \sin(t-\rho\pi)\dd t>0.
\end{equation}
The logarithm and the sine have the same sign on each of
$(0,\rho\pi)$ and $(\rho\pi,a)$. Every zero is therefore an
upward crossing, so there is exactly one.
For $\rho=1$, pairing the two half-waves gives
$I(\mu,1)<0$ for $\mu<1$, and $I(1,1)=0$.

If $\rho>1/2$, put $\beta=2\rho-1\in(0,1)$. Then
\[
 -I(\beta,\rho)
 =\int_0^a t^{\beta-1}\cos(t-\beta\pi/2)\dd t
 >\ga(\beta)-a^{\beta-1}>0.
\]
Here $\ga(\beta)\geq1$ follows by applying Jensen's inequality to
the negative power of a mean-one exponential random variable.
Hence $\nu>2\rho-1$. This bound is automatic for $\rho\leq1/2$.
In addition $\nu<1$.

Rotate the tail in Lemma~\ref{lem:tail} by $e^{-\ii\rho\pi}$.
Since $e^{\ii(a-\rho\pi)}=-1$, its real and imaginary parts are
strictly negative. Write this tail as $-A_1-\ii A_2$, with
$A_1,A_2>0$. The root equation implies
\[
 \ga(\nu)\sin\delta=A_2,\qquad
 D=\ga(\nu)\cos\delta+A_1.
\]
The already established bounds imply $-\pi/2<\delta<\pi/2$;
therefore $\delta>0$, $\nu<2\rho$, and $D>0$.
The real tail gives $A_1=pK(a,p)$, with \eqref{eq:K}.

To strengthen the lower bound to $\nu>\rho$, consider
$G(t)=\ga(t)\sin(t\pi/2)$ for $0<t\leq1$. Duplication and
reflection yield
\[
 G(t)=2^{t-1}\sqrt\pi\,
       \frac{\ga((1+t)/2)}{\ga(1-t/2)}.
\]
The digamma function $\psi$ is concave, since
$\psi''(x)=-2\sum_{m\geq0}(m+x)^{-3}<0$.
The two arguments in the preceding quotient have mean $3/4$;
thus
\[
 (\log G)'(t)\leq\log2+\psi(3/4)
       =-\gamma+\frac{\pi}{2}-2\log2<0.
\]
As $G(1)=1$, $G(t)>1$ for $t<1$. Applying the tail bound at
$\mu=\rho$ gives
$I(\rho,\rho)=-G(\rho)+A_2<0$, because
$A_2<a^{\rho-1}<1$. Consequently $\nu>\rho$.

Finally the moving upper endpoint contributes zero to $I_\rho$,
so $I_\rho=-\pi D$. The implicit function theorem and
\eqref{eq:Imu} give
$\nu'=\pi D/I_\mu>0$.
\end{proof}

\begin{lemma}\label{lem:H}
For $0<\rho<1$ and the critical parameter,
\begin{equation}\label{eq:H}
 H(x):=-\int_0^x t^{-p}\sin(t-\rho\pi)\dd t\geq0
                  \quad(x\geq0).
\end{equation}
The only positive zero is $x=a$.
\end{lemma}
\begin{proof}
The first maximum of $H$ is at $\rho\pi$, and its first minimum
is at $a$, where the root equation gives $H(a)=0$.
Between consecutive minima, the positive half-wave occurs
before the negative half-wave. Pairing points separated by
$\pi$ shows that the first has the larger weight $t^{-p}$.
Every later minimum is therefore strictly larger than the
preceding one. This also proves strict positivity before the
first minimum, except at that minimum.
\end{proof}

\section{A gamma comparison valid for all critical parameters}

We first prove an estimate which has no lower restriction on $\nu$.

\begin{lemma}\label{lem:gamma}
For $0<\nu<1$ and $x>0$ put
\[
 \Delta_\nu(x)=\frac{x^{\nu-1}}{\ga(\nu)}
              -\frac{\ga(x+\nu)}{\ga(\nu)\ga(x+1)}.
\]
Then $\Delta_\nu$ is positive and completely monotone, and
\begin{equation}\label{eq:Delta}
 \Delta_\nu(x)\leq
 \frac{\nu p}{2\ga(\nu)x^{2-\nu}}
       \left(1+\frac{2-\nu}{12x}\right).
\end{equation}
For any positive nonincreasing sequence $d_k\to0$,
\begin{equation}\label{eq:Dirichlet}
 \left|\sum_{k=n+1}^\infty d_ke^{2\ii k\phi}\right|
          \leq\frac{d_{n+1}}{\sin\phi},
                  \qquad 0<\phi\leq\pi/2.
\end{equation}
\end{lemma}
\begin{proof}
The beta integral, with the substitution $t=e^{-s}$, gives
\[
 \Delta_\nu(x)=\frac1{\ga(\nu)\ga(p)}
       \int_0^\infty e^{-xs}
          \left(s^{-\nu}-(e^s-1)^{-\nu}\right)\dd s.
\]
Its kernel is positive. Differentiating under the integral
proves complete monotonicity; domination follows locally in
$x>0$, since the kernel is $O(s^{1-\nu})$ at zero.

For $v\geq0$,
\[
 (1+v^2/3)\sinh v-v\cosh v
   =\sum_{j\geq2}\frac{4j(j-1)v^{2j+1}}{3(2j+1)!}\geq0.
\]
Integration gives $\log(\sinh v/v)\leq v^2/6$. Hence
$\ell(s):=\log((e^s-1)/s)\leq s/2+s^2/24$.
Using $1-e^{-y}\leq y$ for $y\geq0$, the kernel is at most
$\nu s^{-\nu}(s/2+s^2/24)$.
Its Laplace integral is \eqref{eq:Delta}.

For \eqref{eq:Dirichlet}, multiply the series by
$1-e^{2\ii\phi}$ and sum differences. The sum of the absolute
values of the coefficients is at most $2d_{n+1}$, whereas
$|1-e^{2\ii\phi}|=2\sin\phi$.
\end{proof}

\begin{lemma}\label{lem:cell}
Let $w(x)=x^{-p}$ and, for $k\geq1$, set
\begin{align*}
 A_k&=\int_0^{1/2}\int_0^t
   \left(\frac{p}{(k+s)^{p+1}}
          -\frac{p}{(k-t+s)^{p+1}}\right)
       \dd s\,e^{2\ii\phi(k-t)}\dd t,\\
 B_k&=\int_0^{1/2}2\ii\sin(2\phi t)
        \int_0^t\frac{p}{(k+s)^{p+1}}\dd s\,
                         e^{2\ii k\phi}\dd t .
\end{align*}
Then
\begin{equation}\label{eq:cell}
 \int_{k-1/2}^{k+1/2}w(x)e^{2\ii\phi x}\dd x
        =u w(k)e^{2\ii k\phi}-A_k-B_k.
\end{equation}
For $n\geq1$,
\begin{equation}\label{eq:ABtails}
 \left|\sum_{k>n}A_k\right|\leq\frac{p}{24n^{p+1}},
 \qquad
 \left|\sum_{k>n}B_k\right|
       \leq\frac{\phi p}{6\sin\phi\,n^{p+1}}.
\end{equation}
\end{lemma}
\begin{proof}
Split the integral at $k$, write the two values of $w$ as
$w(k)$ plus the integrals of its derivative, and collect the
two exponential factors. This yields \eqref{eq:cell}.
The integrand difference defining $A_k$ has absolute value
at most $p(p+1)t(k-1/2)^{-p-2}$. Consequently
\[
 \sum_{k>n}|A_k|
 \leq\frac{p(p+1)}{24}
             \sum_{k>n}(k-1/2)^{-p-2}
 \leq\frac{p}{24n^{p+1}}.
\]
The last inequality is the midpoint bound for the convex
function $x^{-p-2}$ on the unit intervals in $[n,\infty)$.
For fixed $t$, the positive sequence
$\int_0^tp(k+s)^{-p-1}\dd s$ decreases with $k$.
Apply \eqref{eq:Dirichlet}, and use
$2|\sin(2\phi t)|\leq4\phi t$ and
$\int_0^{1/2}t^2\dd t=1/24$, to obtain the second bound.
\end{proof}

\begin{proposition}\label{prop:comparison}
Let $f$ be either $\sin$ or $\cos$, let $c\in\R$, and put
\begin{align*}
 J(x)&=\int_0^x t^{-p}f(t+c)\dd t,\\
 Q&=f\bigl(\nu(\pi-\phi)/2+c-\pi/2\bigr),\\
 r&=f\bigl(\nu(\pi-2\phi)/2+c\bigr),\qquad
 s=\frac{1-u^p}{\sin\phi}.
\end{align*}
For $n\geq1$ and $0<\phi\leq\pi/2$,
\begin{equation}\label{eq:comparison}
\begin{split}
 2^\nu\phi^{\nu-1}\ga(\nu)
     \sum_{k=0}^n a_k(\nu)f(2k\phi+c)
 \geq {}&\frac{J((2n+1)\phi)}{\sin\phi}\\
 &+\ga(\nu)
     \left(\frac{2Q\sin(\nu\phi/2)}{\sin\phi}-rs\right)
       -A_n-B_n-f_nC_n ,
\end{split}
\end{equation}
where
\begin{equation}\label{eq:errors}
\begin{gathered}
 A_n=\frac{p}{4n(2n\phi)^pu},\qquad
 B_n=\frac{p}{3n(2n\phi)^pu^2},\\
 C_n=\frac{\pi\nu p}{[2(n+1)\phi]^{2-\nu}},\qquad
 f_n=1+\frac{2-\nu}{12(n+1)}.
\end{gathered}
\end{equation}
Moreover
\begin{equation}\label{eq:sbound}
 0\leq s\leq\frac{p\phi}{6u^{\nu+1}}.
\end{equation}
\end{proposition}
\begin{proof}
For $k\geq1$, $a_k(\nu)=k^{-p}/\ga(\nu)-\Delta_\nu(k)$.
The complete binomial series converges at
$e^{2\ii\phi}\ne1$ and has value
$(2\sin\phi)^{-\nu}e^{\ii\nu(\pi/2-\phi)}$.
Sum \eqref{eq:cell} for $k>n$, and subtract this tail from
the complete series. Its integral starts at $n+1/2$.
After the substitution $t=2\phi x$ and multiplication by
$2^\nu\phi^{\nu-1}\ga(\nu)$, the integral contribution is
$J((2n+1)\phi)/\sin\phi$.

The complete-series contribution minus the complete-integral
contribution is exactly
\[
 \frac{\ga(\nu)}{\sin\phi}
 \left[u^p f\bigl(\nu(\pi-2\phi)/2+c\bigr)
                   -f(\nu\pi/2+c)\right].
\]
The elementary sine and cosine difference formulas rewrite it
as the gamma term in \eqref{eq:comparison}.
The first bound in \eqref{eq:ABtails} gives the smaller error
$p/[12n(2n\phi)^pu]$, which we enlarge to $A_n$.
The second gives $B_n$.
By \eqref{eq:Delta} and \eqref{eq:Dirichlet}, the scaled
$\Delta$ tail is at most
\[
 \frac{2\nu p f_n}{u[2(n+1)\phi]^{2-\nu}}\leq f_n C_n,
\]
because $u\geq2/\pi$.
Finally $1-u^p\leq p(1-u)u^{p-1}$ and
$1-u\leq\phi^2/6$ prove \eqref{eq:sbound}.
\end{proof}

\section{Uniform lower envelopes for all large degrees}

\subsection{Geometric prefix regions}

For $j\geq0$,
\begin{align}
 \sum_{k=0}^j\sin((2k+\rho)\phi)
   &=\frac{\sin((j+1)\phi)\sin((j+\rho)\phi)}{\sin\phi},
                                      \label{eq:geomL}\\
 \sum_{k=0}^j\sin(\rho\pi-(2k+\rho)\phi)
   &=\frac{\sin((j+1)\phi)
              \sin(\rho\pi-(j+\rho)\phi)}{\sin\phi}.
                                      \label{eq:geomR}
\end{align}
Since $a_k(\nu)$ decreases strictly when $\nu<1$, Abel
summation of these prefixes gives
\begin{equation}\label{eq:geomregions}
 L_n(\phi)>0\quad\left(0<\phi\leq\frac{\pi}{n+1}\right),
 \qquad
 R_n(\phi)>0\quad\left(0<\phi\leq\frac{\rho\pi}{n+\rho}\right).
\end{equation}
For $n\geq N$, outside the first region, valid lower bounds
for $2n\phi$ and $2(n+1)\phi$ are respectively
\begin{equation}\label{eq:XYL}
 X_L=\max\!\left\{\frac{2N\pi}{N+1},2N\phi\right\},
 \qquad
 Y_L=\max\{2\pi,2(N+1)\phi\}.
\end{equation}
Outside the second region we may use
\begin{equation}\label{eq:XYRmid}
 X_R=\max\!\left\{\frac{2N\rho\pi}{N+\rho},2N\phi\right\},
 \qquad
 Y_R=\max\{2\rho\pi,2(N+1)\phi\}.
\end{equation}

\subsection{The right envelope}

For $R_n$ take $f=\cos$ and
$c=\rho\phi+\pi/2-\rho\pi$ in Proposition~\ref{prop:comparison}.
Put $h=\rho\phi$ and $b_R=a-h$.
The first maximum of $J$ is at $\rho\pi-h$, and its first
minimum is at $b_R$. Rotation of the critical integral gives
$J(a)=-D\sin h$. On $[b_R,a]$, the remaining integral is
positive and at most $b_R^{-p}(1-\cos h)$. Thus
\begin{equation}\label{eq:Jright}
 J(b_R)\geq-D\sin h-b_R^{-p}(1-\cos h).
\end{equation}
All subsequent local minima of $J$ increase by pairing its
negative and positive half-waves. Therefore \eqref{eq:Jright}
is a lower bound for $J(x)$ whenever $x$ is beyond its first
maximum.

For any valid $X_R,Y_R$ as above, define
\begin{align}
 M_R&=\frac{\ga(\nu)\sin\delta-b_R^{-p}}p,\label{eq:MR}\\
 f_N&=1+\frac{2-\nu}{12(N+1)},\label{eq:fN}\\
 E_N(X,Y)&=\frac{1/(4u)+1/(3u^2)}{NX^p}
                     +\frac{f_N\pi\nu}{Y^{2-\nu}}.\label{eq:EN}
\end{align}
The right envelope is
\begin{equation}\label{eq:GRdirect}
\begin{split}
 G_R={}&
 \frac{\ga(\nu)\sigma}{pu}
        \cos(\delta+\sigma\phi/2)\sinc(\sigma\phi/2)\\
 &+\frac{M_R\rho^2\phi}{2u}\sinc^2(\rho\phi/2)
       -\frac{a^{-p-1}\rho}{u}\sinc(\rho\phi)\\
 &-\frac{\ga(\nu)\sin(\delta+\sigma\phi)\phi}{6u^{\nu+1}}
       -E_N(X_R,Y_R).
\end{split}
\end{equation}
Indeed, insert \eqref{eq:D} in \eqref{eq:Jright} and use
\eqref{eq:sbound} in \eqref{eq:comparison}.
The gamma terms collect by
\[
 \sin(\delta+\sigma\phi)-\sin\delta
       =\sigma\phi\cos(\delta+\sigma\phi/2)
                                     \sinc(\sigma\phi/2).
\]
The coefficient of $1-\cos h$ is $pM_R$, and
$K(a,p)<a^{-p-1}$. Also
$0<\delta+\sigma\phi\leq\rho\pi/2$.
Since $n\geq N$, all three errors in
\eqref{eq:comparison} are at most $pE_N(X_R,Y_R)$.
Consequently
\begin{equation}\label{eq:Rbound}
 2^\nu\phi^{\nu-1}\ga(\nu)R_n(\phi)\geq pG_R.
\end{equation}

For computations near $\rho=0$, \eqref{eq:GRdirect} has the
following nonsingular form:
\begin{align}
 \widetilde M_R
   &=\frac{\ga_1\pi(1/q-1/2)\sinc\delta-b_R^{-p}}p,\label{eq:MRstable}\\
 V_R&=\pi(1/q-1/2)+(1-1/q)\phi,\label{eq:VR}\\
 G_R={}&
 \frac{\ga_1(1-1/q)}{pu}
       \cos(\delta+\sigma\phi/2)\sinc(\sigma\phi/2)\notag\\
 &+\frac{\widetilde M_R\rho^2\phi}{2u}\sinc^2(\rho\phi/2)
       -\frac{a^{-p-1}\rho}{u}\sinc(\rho\phi)\notag\\
 &-\frac{\ga_1V_R\sinc(\delta+\sigma\phi)\phi}{6u^{\nu+1}}
       -E_N(X_R,Y_R).\label{eq:GRstable}
\end{align}
This is the same expression: $\ga(\nu)\sigma=\ga_1(1-1/q)$,
$\ga(\nu)\sin\delta=\ga_1\pi(1/q-1/2)\sinc\delta$,
and $\delta+\sigma\phi=\nu V_R$.

\subsection{The left envelope}

For $L_n$ take $f=\sin$, $c=\rho\phi$.
The first minimum of $J$ is at $b_L=2\pi-\rho\phi$.
The complete Fourier integral and Lemma~\ref{lem:tail} give
\begin{equation}\label{eq:Jleft}
 J(b_L)\geq\ga(\nu)\sin(\nu\pi/2+\rho\phi)-b_L^{-p}.
\end{equation}
After its first maximum, every value of $J$ is at least
$J(b_L)$, again by half-wave pairing.
Outside the left region in \eqref{eq:geomregions},
$(2n+1)\phi$ is beyond that maximum.
The complete-integral term in \eqref{eq:Jleft} cancels the
corresponding gamma correction in \eqref{eq:comparison}, leaving
\[
 2^\nu\phi^{\nu-1}\ga(\nu)L_n(\phi)
   \geq\frac{\ga(\nu)\sin(\nu\pi/2-\sigma\phi)u^p-b_L^{-p}}
               {\sin\phi}-pE_N(X_L,Y_L).
\]
Set
\begin{align}
 M_L&=\frac{\ga(\nu)\sin(\nu\pi/2-\sigma\phi)u^p-b_L^{-p}}p,
                                                     \label{eq:ML}\\
 G_L&=M_L-E_N(X_L,Y_L).\label{eq:GL}
\end{align}
If $G_L>0$, then $M_L>0$ and $1/\sin\phi\geq1$ gives
\begin{equation}\label{eq:Lbound}
 2^\nu\phi^{\nu-1}\ga(\nu)L_n(\phi)\geq pG_L>0.
\end{equation}
The stable form used at $\rho=0$ is
\begin{equation}\label{eq:MLstable}
 V_L=\pi/2-(1-1/q)\phi,\qquad
 M_L=\frac{\ga_1V_L\sinc(\nu V_L)u^p-b_L^{-p}}p.
\end{equation}

\section{An exact beta estimate for small parameters}

\begin{lemma}\label{lem:beta}
For $\mu>0$ and $n\geq0$, let
$c_n(\mu)=s_n^\mu(1)$. Then
\begin{equation}\label{eq:beta}
 c_n(\mu)=\frac{\ga(n+\mu+1)}{\ga(1+\mu)\ga(n+1)},\qquad
 \frac{s_n^\mu(z)}{c_n(\mu)}
       =\mu\int_0^1t^{\mu-1}(1-t+tz)^n\dd t .
\end{equation}
For $n\geq1$, at the critical parameter put
$x=2n\phi$, $W_n(x)=R_n(\phi)/c_n(\nu)$ and
$C(x)=\nu x^{-\nu}H(x)$. Then
\begin{equation}\label{eq:variance}
 |W_n(x)-C(x)|
    \leq\frac{\nu x^2}{2n(\nu+1)(\nu+2)}
                  +\frac{\rho x}{2n}
    \leq\frac{\rho(x^2+x)}{2n}.
\end{equation}
\end{lemma}
\begin{proof}
Expand the polynomial in the integral in \eqref{eq:beta}.
The coefficient of $z^k$ is
$\mu\binom nk B(\mu+k,n-k+1)=a_k(\mu)/c_n(\mu)$.
This proves the integral identity and its normalization.

Conditionally on $t$, let $K$ have the binomial distribution
with parameters $n,t$. Then
\[
 W_n(x)=\nu\int_0^1t^{\nu-1}
       \E\sin\left(\rho\pi-\frac{\rho x}{2n}-\frac{xK}{n}\right)
                 \dd t,\qquad
 C(x)=\nu\int_0^1t^{\nu-1}\sin(\rho\pi-xt)\dd t .
\]
The conditional mean of $K/n-t$ is zero and its variance is
$t(1-t)/n$. Taylor's theorem, using $|\sin''|\leq1$, bounds
the contribution from this fluctuation by
$\nu x^2/[2n(\nu+1)(\nu+2)]$.
The extra phase $\rho x/(2n)$ contributes at most its
absolute value. Finally $\nu<2\rho$ and
$(\nu+1)(\nu+2)\geq2$ give the last inequality.
This estimate holds for every $x\geq0$; no asymptotic
expansion in $x/n$ is involved.
\end{proof}

\begin{proposition}\label{prop:precritical}
For $0<\rho\leq2/5$ and $n\geq64$,
$R_n(\phi)>0$ whenever $0<2n\phi\leq2a/3$.
\end{proposition}
\begin{proof}
The strictly certified root bracket at $\rho=2/5$ gives
$\nu(2/5)<3/5$; see Proposition~\ref{prop:certsmall}.
Lemma~\ref{lem:root} then implies $\nu<3/5$ throughout this
range. Put $A=7\pi/5$, so $a\leq A$.

If $0<x\leq\rho\pi$, the Lipschitz bound for the sine gives
$C(x)\geq\sin(\rho\pi)-x\nu/(\nu+1)$.
Since $\rho\pi\leq2\pi/5<\pi/2$,
$\sin(\rho\pi)\geq2\rho$, whence
\[
 \frac{W_n(x)}{\rho}
 \geq 2-\frac{3\pi}{8}
        -\frac{(2\pi/5)^2+2\pi/5}{128}>0.
\]
If $\rho\pi\leq x\leq2a/3$, the root equation gives
\[
 H(x)=\int_x^a t^{-p}\sin(t-\rho\pi)\dd t
      \geq a^{-p}\{1-\cos(a-x)\}.
\]
Here $a/3\leq a-x\leq\pi$, so the braces are at least $1/2$.
As $\nu>\rho$,
$\nu x^{-\nu}a^{-p}=(\nu/a)(a/x)^\nu>\rho/a$.
Therefore
\[
 \frac{W_n(x)}{\rho}
 \geq\frac1{2A}
       -\frac{(2A/3)^2+2A/3}{128}>0 .
\]
The two last numerical expressions have certified lower
bounds greater than $0.7997$ and $0.0236$, respectively.
\end{proof}

For the rest of the right half-circle use
$x=2n\phi\geq2a/3$ and, with $N=64$, take
\begin{equation}\label{eq:XYRsmall}
 X_R=\max\{2a/3,2N\phi\},\qquad
 Y_R=\max\{2a/3+2\phi,2(N+1)\phi\}.
\end{equation}
This scale bound is beyond the first maximum $\rho\pi-\rho\phi$,
so \eqref{eq:Rbound} applies.
The left envelope continues to use \eqref{eq:XYL}.

\begin{proposition}[Small-parameter large-degree certificate]\label{prop:certsmall}
With $N=64$, the envelopes \eqref{eq:GRstable} and
\eqref{eq:GL}, \eqref{eq:MLstable}, using
\eqref{eq:XYRsmall} and \eqref{eq:XYL}, are strictly positive
for every $0<\rho\leq2/5$ and $0\leq\phi\leq\pi/2$
at $\nu=\nu(\rho)$. Also $\nu(2/5)<3/5$.
\end{proposition}
\begin{proof}[Certificate verification]
For $d_0=1/100000$, the nonsingular root equation of
Section~\ref{sec:finite} certifies
\begin{equation}\label{eq:smallq}
 \frac{212}{125}<q(\rho)<\frac{1697}{1000}
                   \qquad(0<\rho\leq d_0).
\end{equation}
Evaluate the stable formulas with the complete parameter
intervals $\rho\in[0,d_0]$, $q\in[212/125,\allowbreak1697/1000]$,
and $\nu=q\rho$. Each side is certified on $64$ angle intervals.
Appendix~\ref{app:small} also gives uniform constant bounds
$G_R>1/8$ and $G_L>1/3$ for this range.

On $[d_0,2/5]$, the file
\file{scale_certificate.json} specifies $244$ rational
parameter intervals. At each endpoint, the bracket
$\nu_-<\nu(\rho)<\nu_+$ is verified by strict opposite signs
of $I$, using the Taylor remainder bounds in
Appendix~\ref{app:roots}. Root monotonicity encloses the
critical curve throughout every interval. The left and right
angle covers have $15616$ cells each. Re-evaluation of the
stable formulas with ball arithmetic gives a strictly
positive lower endpoint on every cell. All adjacent endpoints
and full coverage are checked as exact rational numbers.
The upper bracket at $\rho=2/5$ is smaller than $3/5$.

The alternative verifier \file{check_extension.py} uses
\eqref{eq:GRdirect} and \eqref{eq:ML}, without importing the
generator. It also verifies root signs by a different
integrated Taylor formula. With additional subdivisions
where needed, $246$ parameter cells and $31493$ angle cells
pass. The source, rational input cells and run reports are
included in the supplementary files.
\end{proof}

\begin{corollary}\label{cor:smalllarge}
For $0<\rho\leq2/5$ and $n\geq64$,
$L_n(\phi)>0$ and $R_n(\phi)>0$ for $0<\phi\leq\pi/2$.
\end{corollary}
\begin{proof}
Use \eqref{eq:geomregions} and \eqref{eq:Lbound} for the left
sum. For the right sum, Proposition~\ref{prop:precritical}
covers $2n\phi\leq2a/3$, while Proposition~\ref{prop:certsmall}
and \eqref{eq:Rbound} cover $2n\phi\geq2a/3$.
\end{proof}

\section{The root ratio and the remaining small-parameter degrees}
\label{sec:finite}

\subsection{A normalized root equation}

For $\nu=q\rho$ and $a=(1+\rho)\pi$, set
\begin{align}
 C_0(\nu,a)&=\int_0^1t^{\nu-1}(\cos(at)-1)\dd t,\notag\\
 S_0(\nu,a)&=\int_0^1t^{\nu-1}\sin(at)\dd t.\label{eq:CS}
\end{align}
Both expressions are continuous at $\nu=0$. The equation
for the root, with its removable singularity extracted, is
\begin{equation}\label{eq:F}
\begin{split}
 F(q,\rho)&=\frac{\nu I(\nu,\rho)}{\rho a^\nu}\\
 &=q\cos(\rho\pi)S_0(\nu,a)
       -\pi\sinc(\rho\pi)\{1+\nu C_0(\nu,a)\}.
\end{split}
\end{equation}
At a root,
\begin{equation}\label{eq:Fq}
 F_q=\frac{\nu I_\mu(\nu,\rho)}{a^\nu}>0.
\end{equation}
The strict signs of \eqref{eq:F} at the two constants in
\eqref{eq:smallq}, over $\rho\in[0,d_0]$, prove that bound.
They are evaluated using \eqref{eq:CSseries} below.

An integrated exponential Taylor formula is
\begin{equation}\label{eq:Fexp}
 F(q,\rho)=-\pi\sinc(\rho\pi)
      +q\sum_{k=1}^{63}
        \frac{a^k\sin(k\pi/2-\rho\pi)}{k!(\nu+k)}
      +\mathcal R ,
 \qquad
 |\mathcal R|\leq
       \frac{q a^{64}}{64!(\nu+64)}.
\end{equation}
For fixed $q$, the remainder derivative is bounded by
\begin{equation}\label{eq:Fremder}
 |\partial_\rho\mathcal R|
 \leq \frac{q\pi}{\nu+64}
       \left(\frac{a^{63}}{63!}+\frac{a^{64}}{64!}\right)
          +\frac{q^2a^{64}}{64!(\nu+64)^2}.
\end{equation}
These are derivative bounds for the integral remainder, not
derivatives of a scalar error estimate. They follow by
differentiating the exponential remainder in $a$, its phase
$e^{-\ii\rho\pi}$, and the weight $t^{q\rho-1}$.
The identity
\[
 \sinc'(x)=-\int_0^1t\sin(tx)\dd t
\]
gives, after retaining terms through $x^{47}$, the error
$|x|^{49}/(49!\,51)$ for its derivative series.

\begin{proposition}[Root-ratio certificate]\label{prop:ratio}
The function $q(\rho)=\nu(\rho)/\rho$ is strictly decreasing
on $(0,2/5]$, and
\begin{equation}\label{eq:qzero}
 \lim_{\rho\downarrow0}q(\rho)=
 \kappa:=\frac{\pi}{\Si(\pi)}
       <1.6963819856765,
 \qquad
 \Si(x)=\int_0^x\frac{\sin t}{t}\dd t .
\end{equation}
\end{proposition}
\begin{proof}
Using the derivative of \eqref{eq:Fexp}, together with
\eqref{eq:Fremder}, the file
\file{ratio_monotone_certificate.json} certifies
$F_\rho>0$ on $1240$ rational parameter intervals covering
$[d_0,2/5]$. Critical roots at all endpoints are enclosed by
strictly checked rational brackets. On the remaining interval,
the same strict inequality is checked with
$\rho\in[0,d_0]$ and
$q\in[212/125,\allowbreak1697/1000]$.

For an alternative verification, \file{check_extension.py}
uses the separate sine and cosine integrals
\eqref{eq:CS}, their $\nu$ derivatives, and
\[
 \partial_a C_0(\nu,a)=-S_0(\nu+1,a),\qquad
 \partial_a S_0(\nu,a)=\frac1{\nu+1}+C_0(\nu+1,a).
\]
The derivative remainders are obtained by integrating the
Taylor bounds against $t^{\nu-1}|\log t|$.
All $1240$ parameter intervals and the endpoint interval
give strictly positive lower bounds.
Since \eqref{eq:Fq} holds on the actual root curve,
$q'=-F_\rho/F_q<0$.

By \eqref{eq:rootbounds}, $1<q<2$. The equation
\eqref{eq:F} extends continuously to $\rho=0$ with value
$q\Si(\pi)-\pi$. Thus every accumulation point of $q$ is
$\kappa$, giving the limit in \eqref{eq:qzero}.
The stated rational upper bound is verified by ball
evaluation of $\pi/\Si(\pi)$.
\end{proof}

\subsection{Finite prefix enclosures}

Let $[l,h]\subset[0,2/5]$ be a rational parameter interval.
By Proposition~\ref{prop:ratio}, its critical root ratio is
enclosed by
\begin{equation}\label{eq:qbox}
 q_-=\frac{\nu_-(h)}h,\qquad
 q_+=
 \begin{cases}
  \nu_+(l)/l,&l>0,\\
  1.6963819856765,&l=0.
 \end{cases}
\end{equation}
Here the subscripts denote verified root brackets.
Use the complete intervals
$\rho\in[l,h]$, $q\in[q_-,q_+]$ and $\nu=q\rho$.

For $k\geq1$, the normalized coefficient
\begin{equation}\label{eq:bk}
 b_k=\frac{a_k(\nu)}{\rho}
       =\frac qk\prod_{m=1}^{k-1}\left(1+\frac{\nu}{m}\right)
\end{equation}
is nonsingular at $\rho=0$. The normalized zero terms of
$L_j$ and $R_j$ are respectively
$\phi\sinc(\rho\phi)$ and
$(\pi-\phi)\sinc(\rho(\pi-\phi))$.

For either normalized sum $P_j=L_j/\rho$ or $R_j/\rho$,
a uniform second-derivative bound on an angle interval is
\begin{equation}\label{eq:second}
 |P_j''(\phi)|\leq M_j^{(2)}
       :=\rho+\sum_{k=1}^j b_k(2k+\rho)^2 .
\end{equation}
If $m$ and $r$ are the midpoint and radius of that interval,
Taylor's theorem encloses the entire interval by
\begin{equation}\label{eq:Taylorbox}
 P_j(m)+P_j'(m)[-r,r]
       +[-M_j^{(2)}r^2/2,M_j^{(2)}r^2/2].
\end{equation}
Every quantity in this formula is evaluated on the complete
parameter intervals. It is therefore an interval enclosure
in both variables.

\begin{proposition}[Small-parameter prefix certificate]\label{prop:prefix}
For $0<\rho\leq2/5$, $1\leq j\leq63$, and
$0<\phi\leq\pi/2$, both $L_j(\phi)$ and $R_j(\phi)$
are strictly positive.
\end{proposition}
\begin{proof}[Certificate verification]
The file \file{prefix_monotone_certificate.json} contains
$512$ rational parameter intervals covering $[0,2/5]$.
For each interval, \eqref{eq:qbox} is re-established from
strict root signs and Proposition~\ref{prop:ratio}.
For every angle cell, \eqref{eq:Taylorbox} is evaluated for
all $j=1,\ldots,63$ at once. The left angle range
$\phi\leq\pi/64$ is covered by \eqref{eq:geomregions};
all other cells have strictly positive Taylor lower bounds.
The saved covers contain $70036$ left cells and $107910$
right cells.

The alternative verifier constructs $b_k$ by the Euler
product \eqref{eq:bk}, and the sine phases by complex powers
of $e^{2\ii\phi}$ or $e^{-2\ii\phi}$ instead of separate
sine evaluations. It rechecks the root and ratio-derivative
certificates before checking the prefixes.
Its accepted cover has $512$ parameter cells and $178630$
angle cells. Exact rational coverage is checked independently
of the saved lower-bound strings.
\end{proof}

\section{The middle parameter range and the upper endpoint}

\subsection{The compact middle range}

For $2/5\leq\rho\leq999/1000$ take $N=8$ and use
\eqref{eq:XYL}, \eqref{eq:XYRmid} in the large-degree envelopes.
An equivalent expression is useful for the first term of
$G_R$. Write
\[
 \eps=1-\rho,\qquad \lambda=\nu-2\rho+1=2\eps-p .
\]
Then
\begin{equation}\label{eq:GRupperfirst}
\begin{split}
 &\frac{\ga(\nu)\sigma}{pu}
          \cos(\delta+\sigma\phi/2)\sinc(\sigma\phi/2)\\
 &\quad =
 \frac{\ga(\nu)}{2u}
   \left(\frac{\sigma\eps}{p}\pi
              +\frac{\sigma^2}{p}(\pi-\phi)\right)
   \sinc\left(\frac{\eps\pi+\sigma(\pi-\phi)}2\right)
   \sinc(\sigma\phi/2).
\end{split}
\end{equation}
Indeed $\delta=\pi/2-\lambda\pi/2$, and
$\lambda\pi-\sigma\phi=\eps\pi+\sigma(\pi-\phi)$.
For subtractive gamma terms the generator uses
$e^x-1$ to evaluate, for example,
\[
 M_R=\frac{\exp\{\log\ga(\nu)+\log\cos(\lambda\pi/2)\}
                   -\exp\{-p\log b_R\}}p .
\]

\begin{proposition}[Middle-parameter certificate]\label{prop:middle}
For $2/5\leq\rho\leq999/1000$, all critical boundary sums
$L_n$ and $R_n$ are strictly positive on $(0,\pi/2]$.
\end{proposition}
\begin{proof}[Certificate verification]
For $n\geq8$, the geometric regions in
\eqref{eq:geomregions} and the inequalities
$G_L>0$, $G_R>0$ suffice.
The file \file{strong_band_certificate.json} provides $262$
parameter intervals, with $263$ verified endpoint root
brackets. On these intervals it checks $8384$ left envelope
cells and $8542$ right envelope cells.

The remaining degrees $1,\ldots,7$ are certified on full
angle covers $\theta/\pi\in[0,1]$, with $28613$ cells.
The two geometric regions are used wherever they apply.
On all other cells, the actual $T_j$ is enclosed directly;
near $\rho=1$ the difference identity \eqref{eq:prefixid}
below is used instead, to reduce cancellation.

The program \file{check_strong_band.py}, which does not
import the generator, checks root signs with the integrated
exponential Taylor formula. It evaluates
\eqref{eq:GRdirect} and \eqref{eq:ML}, and uses alternate
prefix formulas. After additional interval subdivisions,
$274$ parameter cells and $81337$ angle cells pass
($8768$ left, $9032$ right, and $63537$ prefix cells).
There are $275$ verified root brackets including the
new subdivision endpoints. All covers are exact rational
covers.
\end{proof}

\subsection{A quadratic root bound near \texorpdfstring{$\rho=1$}{rho = 1}}

Let
\begin{equation}\label{eq:upperconstants}
 d=10^{-3},\qquad c_-=\frac{199}{100},\qquad
 c_+=\frac{205}{100},\qquad P=c_+d^2 .
\end{equation}

\begin{lemma}\label{lem:upperroot}
For $0<\eps=1-\rho\leq d$,
\begin{equation}\label{eq:quadratic}
 c_-\eps^2<p=1-\nu(\rho)<c_+\eps^2 .
\end{equation}
\end{lemma}
\begin{proof}
Put $m_0=1-3d^2$ and $a_0=2\pi$. Define
\begin{align*}
 M_1&=m_0^{-2}+a_0\log a_0-a_0+1,\\
 M_2&=2m_0^{-3}
       +a_0\{(\log a_0)^2-2\log a_0+2\}-2,\\
 J_0&=\gamma+\log a_0-\Ci(a_0)
        =I_\mu(1,1).
\end{align*}
Here
$\Ci(x)=\gamma+\log x+\int_0^x(\cos t-1)\dd t/t$.
Integrating the absolute values of $\log t$ and $(\log t)^2$
on $(0,1)$ with weight $t^{m_0-1}$, and on $(1,a_0)$
with weight at most $1$, gives
\[
 |I_{\mu\rho}|\leq\pi M_1,\qquad
 |I_{\mu\mu}|\leq M_2
\]
on $m_0\leq\mu\leq1$, $1-d\leq\rho\leq1$.
The moving endpoint in $I_{\mu\rho}$ is again zero.
Consequently
\[
 |I_\mu(\mu,\rho)-J_0|
           \leq\pi M_1d+3M_2d^2 .
\]
Ball arithmetic verifies that this entire derivative
interval lies in $(241/100,247/100)$.
Also
\[
 \frac{I(1,\rho)}{\eps^2}
 =\frac{1-\cos(\pi\eps)}{\eps^2}
 \in\left[\frac{\pi^2}{2}-\frac{\pi^4d^2}{24},
                         \frac{\pi^2}{2}\right].
\]
The strict inequalities
\[
 \frac{\pi^2}{2}-\frac{\pi^4d^2}{24}
             -c_-\frac{247}{100}>0,\qquad
 \frac{\pi^2}{2}-c_+\frac{241}{100}<0
\]
are verified directly. The mean value theorem now gives
opposite signs at $\mu=1-c_-\eps^2$ and
$\mu=1-c_+\eps^2$. Uniqueness of the root proves
\eqref{eq:quadratic}. These constant checks are included in
\file{certify_endpoints.py}.
\end{proof}

\subsection{Removing the singular quotients in the envelopes}

For $0<\eps\leq d$ set
\[
 \eta_1=\frac{\sigma\eps}{p},\qquad
 \eta_2=\frac{\sigma^2}{p},\qquad
 \eta_w=\frac{\lambda^2}{p}.
\]
As $p=c\eps^2$ with $c_-<c<c_+$, the following nonsingular
intervals enclose them:
\begin{equation}\label{eq:etas}
\begin{aligned}
 \eta_1&\in[\,1/c_+-d,\ 1/c_-\,],\\
 \eta_2&\in[\,(1-c_+d)^2/c_+,\ 1/c_-\,],\\
 \eta_w&\in[\,(2-c_+d)^2/c_+,\ 4/c_-\,].
\end{aligned}
\end{equation}
The continuous gamma quotient has the enclosure
\begin{equation}\label{eq:gquot}
 g(p):=\frac{\log\ga(1-p)}p
       \in\left[\gamma,\ \gamma+\frac{\pi^2P}{12(1-P)}\right].
\end{equation}
This follows from
$\log\ga(1-p)=\gamma p+\sum_{k\geq2}\zeta(k)p^k/k$,
with $\zeta(k)\leq\zeta(2)$ and $1/k\leq1/2$.

For $0\leq x\leq x_0<\pi/2$,
$x\leq\tan x\leq x\sec^2x_0$, and hence
$-x^2\sec^2x_0/2\leq\log\cos x\leq-x^2/2$.
Define the exact logarithmic quotient by
\[
 \ell_R=g(p)+\log b_R+\frac{\log\cos(\lambda\pi/2)}p.
\]
It is enclosed by
$g(p)+\log b_R-\eta_w\pi^2[1,\sec^2(\pi d)]/8$.
The difference quotient in \eqref{eq:MR} is enclosed through
the exact identity
\begin{equation}\label{eq:exprelR}
 M_R=b_R^{-p}\ell_R\exprel(p\ell_R),\qquad
 \exprel(x)=\int_0^1e^{tx}\dd t.
\end{equation}
For interval $x$, $\exprel(x)$ is enclosed by
$\exp(\operatorname{hull}(0,x))$.

For the left quotient put
$\alpha=p\pi/2+\sigma\phi$. Then
\[
 \frac{\alpha^2}{p}
       =\frac{p\pi^2}{4}+\sigma\pi\phi+\eta_2\phi^2 .
\]
The exact quotient
$\ell_L=g(p)+\log b_L+\log u+\log\cos\alpha/p$
can be written as
\[
 \ell_L=g(p)+\log b_L+\log u
       -\frac{\xi_L}{2}
          \left(\frac{p\pi^2}{4}+\sigma\pi\phi+\eta_2\phi^2\right),
 \qquad
 \xi_L\in[1,\sec^2(P\pi/2+d\pi/2)] .
\]
As in \eqref{eq:exprelR},
\begin{equation}\label{eq:exprelL}
 M_L=b_L^{-p}\ell_L\exprel(p\ell_L).
\end{equation}
Together with \eqref{eq:GRupperfirst}, these formulas contain
no division by a parameter interval that includes zero.

\begin{proposition}[Upper-endpoint large-degree certificate]\label{prop:upperlarge}
For $999/1000\leq\rho<1$, $n\geq8$, and
$0<\phi\leq\pi/2$, both critical boundary sums are strictly
positive.
\end{proposition}
\begin{proof}[Certificate verification]
Use $N=8$ and \eqref{eq:XYL}, \eqref{eq:XYRmid}.
Enclose $\eps\in[0,d]$, $p\in[0,P]$,
$\sigma\in[0,d]$ and $\lambda\in[0,2d]$, with the quotient
intervals \eqref{eq:etas}. Evaluate the right envelope with
\eqref{eq:GRupperfirst}, \eqref{eq:exprelR}, and the left
envelope with \eqref{eq:exprelL}.

The program \file{certify_strong_endpoint.py} checks the
$2048$ angle cells
$[j\pi/4096,(j+1)\pi/4096]$, $0\leq j<2048$.
The complete parameter interval is used in every cell.
It certifies $G_R>1/4$ and $G_L>3/4$, with smallest saved
lower endpoints greater than $0.3764$ and $0.7892$.
The geometric regions and
\eqref{eq:Rbound}, \eqref{eq:Lbound} complete the argument.
\end{proof}

\subsection{The seven endpoint prefixes}

For $1\leq j\leq7$ and $0<\theta<\pi$, let
$V_j(\theta)=2\sin\theta\,T_j(\rho,\nu,\theta)$ and
$h=\eps\theta$.
The recurrence $a_k-a_{k-1}=-p\,a_{k-1}/k$ gives the exact
difference identity
\begin{equation}\label{eq:prefixid}
\begin{split}
 V_j(\theta)={}&a_j\{\cos h-\cos(2(j+1)\theta-h)\}\\
 &+p\sum_{k=1}^j\frac{a_{k-1}}k
                      \{\cos h-\cos(2k\theta-h)\}.
\end{split}
\end{equation}
Write $\mathfrak h_j=\sum_{k=1}^j1/k$ and
$d_j=\pi/[100(j+1)]$. Since
$a_j=\prod_{k=1}^j(1-p/k)$, one has
$1-P\mathfrak h_j\leq a_k\leq1$ for $k\leq j$.

Away from all neighborhoods
$|\theta-m\pi/(j+1)|\leq d_j$, $0\leq m\leq j+1$, the main
cosine difference in \eqref{eq:prefixid} is at least
\[
 \cos(\pi d)-\cos\bigl(\pi(2/100-d)\bigr).
\]
Thus
\begin{equation}\label{eq:prefixaway}
\begin{split}
 V_j(\theta)\geq{}&
 (1-P\mathfrak h_j)
       \{\cos(\pi d)-\cos(\pi(2/100-d))\}
             -2P\mathfrak h_j>\frac1{1000}.
\end{split}
\end{equation}
The last seven constant inequalities are verified by ball
arithmetic.

On an internal neighborhood put
\[
 H_j(\theta)=\sum_{k=1}^j\frac{1-\cos(2k\theta)}k,\qquad
 B_j(\theta)=(1-P\mathfrak h_j)\cos(\pi d)H_j(\theta)
                              -\pi d\,\mathfrak h_j .
\]
The main term in \eqref{eq:prefixid} is at least
$-\eps^2\theta^2/2$.
Expanding each smaller difference into its sine and cosine
parts gives a lower bound $pB_j(\theta)$ for the sum of
the remaining terms. For every one of these $28$ internal
neighborhoods, direct interval evaluation verifies
$B_j>0$ and
\begin{equation}\label{eq:prefixnear}
 \frac{V_j(\theta)}{\eps^2}
           \geq-\frac{\theta^2}{2}+c_- B_j(\theta)>\frac1{10}.
\end{equation}
Indeed $p/\eps^2>c_-$ by \eqref{eq:quadratic}.
The neighborhood of $0$ is contained in the left region of
\eqref{eq:geomregions}. The neighborhood of $\pi$ is contained
in the right region, because
$(j+\rho)d_j\leq(j+1)d_j=\pi/100<\rho\pi$.
All these checks are contained in
\file{certify_strong_endpoint.py}. We have proved:

\begin{proposition}\label{prop:upperprefix}
For $999/1000\leq\rho<1$ and $1\leq j\leq7$,
$T_j(\rho,\nu,\theta)>0$ for $0<\theta<\pi$.
\end{proposition}

\section{Completion, sharpness and the weak conjecture}

\begin{proof}[Proof of Theorem~\ref{thm:main}]
For $0<\rho\leq2/5$, Corollary~\ref{cor:smalllarge} covers
every $n\geq64$, and Proposition~\ref{prop:prefix} covers
the remaining positive degrees.
For $2/5\leq\rho\leq999/1000$, use
Proposition~\ref{prop:middle}.
For $999/1000\leq\rho<1$, use
Propositions~\ref{prop:upperlarge} and \ref{prop:upperprefix}.
The critical sum of degree zero is
$T_0(\rho,\nu,\theta)=\sin(\rho\theta)>0$ on $(0,\pi)$.
The two half-circle sums together cover the full circle.
At $\theta=0,\pi$ they are nonnegative by continuity.
Lemma~\ref{lem:boundary} proves the result for all
$0<\mu\leq\nu$.

When $\rho=1$, Lemma~\ref{lem:root} gives $\nu=1$, and
$(1-z)s_n^1(z)=1-z^{n+1}$ has strictly positive real part
in $\D$. For $0<\mu\leq1$, the Abel identity
\eqref{eq:abelmu}, or the corresponding identity for
$s_n^\mu$, gives the same conclusion. This is
\eqref{eq:strong} with $\rho=1$.

It remains to prove optimality. Fix $\rho<1$ and
$\nu<\mu\leq1$. Uniqueness of the upward-crossing root gives
$I(\mu,\rho)>0$. Apply the exact beta representation with
parameter $\mu$ at $\phi=a/(2n)$. The first variance bound
in \eqref{eq:variance}, with $\mu$ replacing $\nu$, gives
an error at most
$\mu a^2/[2n(\mu+1)(\mu+2)]+\rho a/(2n)$. Consequently
\begin{equation}\label{eq:sharpstrong}
 \frac{T_n(\rho,\mu,\pi-a/(2n))}{c_n(\mu)}
        \longrightarrow
           -\mu a^{-\mu}I(\mu,\rho)<0 .
\end{equation}
For all sufficiently large $n$, one of the boundary rotated
real parts in \eqref{eq:rot1}--\eqref{eq:rot2} is negative.
The boundary point is different from $1$, so continuity
gives a nearby point in $\D$ with the same violation.
For any $\mu>1$, $s_1^\mu(-1/\mu)=0$ and $-1/\mu\in\D$;
the sector in \eqref{eq:strong} excludes zero. This also
settles optimality when $\rho=1$.
\end{proof}

\begin{proof}[Proof of Corollary~\ref{cor:weak}]
Let $F(z)=(1-z)^\rho s_n^\mu(z)$. Theorem~\ref{thm:main}
gives $|\arg F(z)|<\rho\pi/2$.
Since $|\arg(1-z)|<\pi/2$ in $\D$,
\[
 \left|\arg\bigl((1-z)^{\rho-1}F(z)\bigr)\right|
       \leq|\arg F(z)|+(1-\rho)|\arg(1-z)|<\frac\pi2 .
\]
This proves \eqref{eq:weak}.

For optimality when $\rho<1$ and $\nu<\mu\leq1$, put
$\beta=2\rho-1$ and define
\[
 U_n(\phi)=\sum_{k=0}^n a_k(\mu)
                    \cos\bigl(2k\phi+\beta(\phi-\pi/2)\bigr).
\]
At $z=e^{2\ii\phi}$ the real part in \eqref{eq:weak} is
$(2\sin\phi)^\beta U_n(\phi)$.
At $\phi=a/(2n)$, the beta representation and the cosine
variance estimate give an error at most
$\mu a^2/[2n(\mu+1)(\mu+2)]+|\beta|a/(2n)$. Hence
\[
 \frac{U_n(a/(2n))}{c_n(\mu)}
 \longrightarrow
 \mu\int_0^1t^{\mu-1}\cos(at-\beta\pi/2)\dd t
       =-\mu a^{-\mu}I(\mu,\rho)<0 .
\]
Again this produces an interior violation for some degree.
For $\mu>1$, the zero of $s_1^\mu$ gives a real part equal
to zero at an interior point. The case $\rho=1$ follows
from the same zero argument.
\end{proof}

\section{Verification and supplementary files}\label{sec:repro}

All certified numerical inequalities use real or complex ball
arithmetic at $256$ bits, through Arb \cite{Arb} and
\file{python-flint}. Every input parameter endpoint and angle
endpoint is stored as a rational number. Transcendental
constants and functions are evaluated with enclosing balls.
A cell is accepted only when its whole output ball has a
strictly positive lower endpoint, or when one of the proved
geometric regions applies.

Root approximations used to propose rational brackets have
no role in their acceptance: strict opposite signs of the
integral enclosure are required.
Similarly, recorded decimal lower bounds are reports rather
than inputs to the verifiers. The verifiers recompute signs
and check rational adjacency and endpoints.
The main numerical tasks are recorded in
Table~\ref{tab:certs}.

\begin{table}[htbp]
\centering
\caption{Continuous certificates in the proof.
The derivative row additionally has one enclosing interval
at $\rho=0$.}\label{tab:certs}
\small
\begin{tabular}{@{}p{.29\textwidth}p{.21\textwidth}rr@{}}
\toprule
Task & Parameter range & Parameter cells & Angle cells\\
\midrule
Small-parameter envelopes &$[d_0,2/5]$&244&31232\\
Endpoint envelopes &$[0,d_0]$&1&128\\
Root-ratio derivative &$[d_0,2/5]$&1240&---\\
Degrees $1,\ldots,63$ &$[0,2/5]$&512&177946\\
Middle envelopes &$[2/5,999/1000]$&262&16926\\
Middle degrees $1,\ldots,7$ &$[2/5,999/1000]$&262&28613\\
Upper envelopes &$[999/1000,1]$&1&2048\\
\bottomrule
\end{tabular}
\end{table}

The alternative verifiers use the direct
gamma formulas, separate trigonometric integrals or integrated
exponential formulas, and complex phase products.
They are alternative implementations of the finite
inequalities; the analytic deductions are given in the
preceding sections.
Adaptive subdivision may change the number of accepted cells
without changing the verified parameter range.

The supplementary archive preserves the three source directories
\file{kr_strong_scale}, \file{kr_strong_assessment}, and
\file{kr_continuous_interval}, so their relative imports remain
unchanged. To verify the supplied certificates from the
archive root, run
\begin{verbatim}
python kr_strong_scale/verify_full_strong.py
\end{verbatim}
This command rechecks the small range, the middle range and
the regularized upper endpoint. It requires only
\file{python-flint}; the NumPy and SciPy packages are used
only by the generators that propose new root brackets.
The tested versions are Python $3.12.14$,
\file{python-flint} $0.9.0$, NumPy $2.3.5$ and SciPy $1.17.0$.

The individual generators can also be rerun:
\begin{verbatim}
python kr_strong_scale/certify_scale.py --N 64
python kr_strong_scale/certify_ratio_monotone.py
python kr_strong_scale/certify_prefixes.py --method monotone
python kr_strong_assessment/certify_strong_band.py
python kr_strong_assessment/certify_strong_endpoint.py
\end{verbatim}
The supplied \file{manifest.json} records SHA-256 digests of
the manuscript, source code, rational covers and run reports.
Verification rewrites reports containing elapsed times, so
their post-verification digests can differ from the original
archive digests.

\section{A sharp convolution consequence for starlike functions}
\label{sec:starlike-convolution}

\subsection{Normalization and convolution tools}
Let $\mathcal A$ be the class of analytic functions
$f(z)=z+\sum_{j\geq2}a_jz^j$ in $\D$, and write
\[
 S_{n+1}f(z)=z+\sum_{k=1}^n a_{k+1}z^{k+1},\qquad n\geq0.
\]
The Hadamard product is
$(f*g)(z)=z+\sum_{j\geq2}a_jb_jz^j$ when
$g(z)=z+\sum_{j\geq2}b_jz^j$.
For $0\leq\alpha<1$, denote by $\mathcal S^*(\alpha)$ the normalized
starlike functions of order $\alpha$, characterized by
$\Re(zf'(z)/f(z))>\alpha$. Put
\[
 K_\eta(z)=\frac{z}{(1-z)^\eta},\qquad
 \mathcal R_\alpha
 =\{A\in\mathcal A:A*K_{2-2\alpha}\in\mathcal S^*(\alpha)\}.
\]
These are the prestarlike classes. We shall distinguish the two sectors
\begin{equation}\label{eq:two-sectors}
 \Sigma_\rho=\{w\ne0:|\Arg w|<\rho\pi/2\},\qquad
 \Theta_\rho=\{w\ne0:|\Arg w|<\rho\pi\},
 \quad 0<\rho\leq1.
\end{equation}
Here $\Arg$ is the principal argument. The sector $\Sigma_\rho$ is
convex, whereas $\Theta_\rho$ is nonconvex when $\rho>1/2$.
In particular, $\Theta_1=\mathbb C\setminus(-\infty,0]$.
For $0<\tau\leq2$ let
\[
 h_\tau(z)=\left(\frac{1+z}{1-z}\right)^\tau,
\]
with the branch taking the value $1$ at $0$. It maps $\D$ conformally
onto the sector of half-angle $\tau\pi/2$. Thus
$h_\rho(\D)=\Sigma_\rho$ and $h_{2\rho}(\D)=\Theta_\rho$.

We use two classical facts about prestarlike functions:
\begin{equation}\label{eq:prestarlike-inclusion}
 \mathcal R_\alpha\subseteq\mathcal R_\beta
 \quad(0\leq\alpha<\beta<1),
\end{equation}
and, for $A\in\mathcal R_\alpha$, $G\in\mathcal S^*(\alpha)$ and
analytic $h$ in $\D$,
\begin{equation}\label{eq:convolution-hull}
 \frac{A*(Gh)}{A*G}(\D)
 \subset\overline{\operatorname{co}}\,h(\D).
\end{equation}
In \eqref{eq:convolution-hull}, $Gh$ is ordinary multiplication.
The inclusion \eqref{eq:prestarlike-inclusion} is recorded in
\cite[Lemma~1.1]{RSS}. The convolution theorem is classical
\cite{Ruscheweyh}; a directly accessible statement is
\cite[Theorem~1.1]{ABR}. We record the denominator and strictness
details needed here.

\begin{lemma}[Strict sector transfer]\label{lem:sector-transfer}
Suppose $0\leq\alpha<1$, $A\in\mathcal R_\alpha$ and
$G\in\mathcal S^*(\alpha)$. Then $(A*G)/z$ is analytic and zero-free
in $\D$. If $h$ is analytic, $h(0)=1$ and
$h(\D)\subset\Sigma_\tau$ for $0<\tau\leq1$, the quotient
\[
 Q=\frac{A*(Gh)}{A*G}
\]
extends analytically with $Q(0)=1$, and $Q\prec h_\tau$.
\end{lemma}
\begin{proof}
For clarity, the convex-hull theorem itself implies the needed
zero-free assertion, even if initially interpreted wherever its quotient
is defined. Put $B=A*G$ and suppose $B(z_0)=0$ at a nonzero
$z_0\in\D$. Choose $s\in(0,1)$ with $|z_0|<s^2$, and set
$A_s(z)=A(sz)/s$, $G_s(z)=G(sz)/s$.
Both retain their respective class memberships. The function
$h_s(z)=z/G_s(z)$, with value $1$ at $0$, is analytic and bounded
on a neighborhood of $\overline\D$, since $G_s/z$ is zero-free there.
As $G_sh_s=z$, \eqref{eq:convolution-hull} would make
\[
 \frac{z}{A_s*G_s(z)}=\frac{z}{B(s^2z)/s^2}
\]
bounded wherever defined. Its pole at $z_0/s^2$ is a contradiction.
Also $B(z)=z+O(z^2)$, so the asserted normalization and removable
value of $Q$ follow from the first coefficients.

Let $\vartheta=\tau\pi/2$. The closed convex hull in
\eqref{eq:convolution-hull} lies in the intersection of the two
closed half-planes
\[
 \Re\bigl(e^{\pm\ii(\pi/2-\vartheta)}Q\bigr)\geq0.
\]
Both harmonic functions have value $\sin\vartheta>0$ at $0$.
The strong minimum principle makes both inequalities strict in
$\D$, so $Q(\D)\subset\Sigma_\tau$. Since $Q(0)=1$, conformal
inversion of $h_\tau$ gives the subordination.
\end{proof}

For $\mu>0$ define the inverse coefficient kernel and the
hypergeometric kernel by
\begin{equation}\label{eq:inverse-phi}
 C_\mu(z)=z\sum_{k\geq0}\frac{k!}{(\mu)_k}z^k,
 \qquad
 \Phi_{\rho,\mu}(z)
 =z\,{}_2F_1(\rho,1;\mu;z)
 =z\sum_{k\geq0}\frac{(\rho)_k}{(\mu)_k}z^k.
\end{equation}
For fixed positive parameters their coefficients have only polynomial
growth, by the gamma-ratio asymptotics. Thus these kernels are analytic
in $\D$, and the same is true of $A=C_\mu*f$ for $f\in\mathcal A$.
If $f\in\mathcal S^*(1-\mu/2)$ with $0<\mu\leq1$, then
$A\in\mathcal R_{1-\mu/2}$ and coefficient comparison gives
\begin{equation}\label{eq:three-hadamard-identities}
 A*K_\mu=f,\qquad A*K_\rho=\Phi_{\rho,\mu}*f,\qquad
 A*(zs_n^\mu)=S_{n+1}f.
\end{equation}

\subsection{The convolution theorem and its sharpness}
\begin{theorem}\label{thm:starlike-convolution}
Let $0<\rho\leq1$, $0<\mu\leq\nu(\rho)$ and
$f\in\mathcal S^*(1-\mu/2)$. The function
$(\Phi_{\rho,\mu}*f)/z$ is analytic and zero-free in $\D$, and
for every $n\geq0$,
\begin{equation}\label{eq:starlike-convolution}
 \frac{S_{n+1}f}{\Phi_{\rho,\mu}*f}\prec h_\rho.
\end{equation}
The quotient has removable value $1$ at $0$.
For fixed $\rho$, the parameter range is optimal among
$0<\mu\leq1$. For each fixed admissible $\mu$ and each fixed
$n\geq0$, the uniform half-angle $\rho\pi/2$ cannot be decreased.
\end{theorem}
\begin{proof}
Set $A=C_\mu*f$ and $\gamma=1-\mu/2$.
If $\rho\leq\mu$, then
\[
 \Re\frac{zK_\rho'(z)}{K_\rho(z)}
 =1+\rho\Re\frac{z}{1-z}>1-\rho/2\geq\gamma,
\]
so $K_\rho\in\mathcal S^*(\gamma)$. Theorem~\ref{thm:main}
puts $H=(1-z)^\rho s_n^\mu$ in $\Sigma_\rho$.
The ordinary product $K_\rho H=zs_n^\mu$, together with
\eqref{eq:three-hadamard-identities} and
Lemma~\ref{lem:sector-transfer}, proves
\eqref{eq:starlike-convolution} and the denominator assertion.

Now suppose $\mu<\rho$. In this branch we use $K_\mu$ in place
of $K_\rho$. Lemma~\ref{lem:root} gives $\mu<\nu(\mu)$, since
$\mu<1$. Applying Theorem~\ref{thm:main} at the diagonal parameter
$\rho=\mu$, followed by Lemma~\ref{lem:sector-transfer}, yields
\[
 U=\frac{S_{n+1}f}{f}\prec h_\mu.
\]
The function $(1-z)^{-(\rho-\mu)}$ lies in
$\Sigma_{\rho-\mu}$, because $|\arg(1-z)|<\pi/2$.
Its ordinary product with $K_\mu$ is $K_\rho$, so the same lemma
also yields
\[
 V=\frac{\Phi_{\rho,\mu}*f}{f}\prec h_{\rho-\mu}.
\]
The functions $U,V$ are analytic and zero-free, with value $1$ at
$0$. Since $f/z$ is zero-free, this proves the denominator assertion.
Their analytic logarithms normalized to vanish at $0$ satisfy
\[
 |\Im(\log U-\log V)|
 <\frac{\mu\pi}{2}+\frac{(\rho-\mu)\pi}{2}
 =\frac{\rho\pi}{2}.
\]
Thus $U/V$ maps $\D$ into $\Sigma_\rho$, giving
\eqref{eq:starlike-convolution} in this branch as well.

For sharpness take $f=K_\mu$. Its logarithmic derivative shows
$K_\mu\in\mathcal S^*(1-\mu/2)$, and
\begin{equation}\label{eq:convolution-extremal}
 \Phi_{\rho,\mu}*K_\mu=K_\rho,
 \qquad
 \frac{S_{n+1}K_\mu}{\Phi_{\rho,\mu}*K_\mu}
 =(1-z)^\rho s_n^\mu.
\end{equation}
The parameter optimality in Theorem~\ref{thm:main} therefore
transfers directly. For angular sharpness, fix $n$ and admissible
$\mu$, and let $z_t=1-t^2+\ii t$, $0<t<1$.
Here $|z_t|^2=1-t^2+t^4<1$,
$s_n^\mu(z_t)\to s_n^\mu(1)>0$, and
$\arg(1-z_t)\to-\pi/2$. The argument in
\eqref{eq:convolution-extremal} tends to $-\rho\pi/2$.
\end{proof}

Theorem~\ref{thm:starlike-convolution} is the full-parameter form of
the starlike consequence identified in \cite[after Theorem~4]{SS}
with Conjecture~3 of Koumandos--Ruscheweyh. The case $\rho=1/4$
is \cite[Theorem~3]{KL}, and $\rho=1/3$ is
\cite[Theorem~4]{SS}. The proof above spells out the two parameter
branches separately.
An extension to $1<\mu\leq2$ in the usual nonnegative-order
starlike classes is also impossible: for $f=K_\mu$, the $n=1$
quotient in \eqref{eq:convolution-extremal} vanishes at
$z=-1/\mu\in\D$.

\section{Optimal sectors for partial sums of starlike functions}
\label{sec:optimal-starlike}

\subsection{An uncorrected binomial sharpness lemma}
We first isolate the scalar necessity argument. The sharpness of the
corrected expression in Theorem~\ref{thm:main} alone does not establish
the sharpness of $s_n^\mu$ in the larger sector $\Theta_\rho$.

\begin{lemma}[Interior scaling]\label{lem:interior-scaling}
For fixed $\mu>0$ and $a>0$, set
$z_n=\exp(-n^{-2}+\ii a/n)$, $n\geq1$. Then
\begin{equation}\label{eq:interior-scaling}
 n^{-\mu}s_n^\mu(z_n)
 \longrightarrow\frac{F_\mu(a)}{\ga(\mu)a^\mu},
 \qquad F_\mu(a)=\int_0^a t^{\mu-1}e^{\ii t}\dd t.
\end{equation}
\end{lemma}
\begin{proof}
Write $c_k=(\mu)_k/k!$. For $k\geq1$ the fixed-parameter
gamma-ratio asymptotics give
$c_k=k^{\mu-1}(1+O_\mu(k^{-1}))/\ga(\mu)$.
The singular endpoint is controlled by the exact identity
\begin{equation}\label{eq:binomial-coefficient-sum}
 \sum_{k=0}^m c_k=\frac{(\mu+1)_m}{m!},
\end{equation}
which follows by induction. For fixed $0<\eps<1$, it implies
\[
 \limsup_{n\to\infty}
 \left|n^{-\mu}\sum_{k=0}^{\lfloor\eps n\rfloor}c_kz_n^k\right|
 \leq\frac{\eps^\mu}{\ga(\mu+1)}.
\]
On $\eps n<k\leq n$ the coefficient asymptotics are uniform, while
$e^{-k/n^2}\to1$ uniformly. This part of the sum therefore converges
to $\ga(\mu)^{-1}\int_\eps^1u^{\mu-1}e^{\ii au}\dd u$.
The omitted integral has modulus at most
$\eps^\mu/\ga(\mu+1)$. Letting $\eps\downarrow0$ proves the
limit, including $0<\mu<1$. Substitution $t=au$ gives the factor
$a^{-\mu}$ in \eqref{eq:interior-scaling}.
\end{proof}

\begin{lemma}[Sharp uncorrected binomial sectors]
\label{lem:uncorrected-binomial}
For fixed $0<\rho\leq1$ and $\mu>0$, the following are equivalent:
\begin{enumerate}
\item $s_n^\mu(z)\in\Theta_\rho$ for every $n\geq0$ and $z\in\D$;
\item $0<\mu\leq\nu(\rho)$.
\end{enumerate}
At $\mu=\nu(\rho)$ the uniform half-angle $\rho\pi$ is optimal.
The necessity and the angular lower bound do not use the sufficient
part of Theorem~\ref{thm:main}.
\end{lemma}
\begin{proof}
For sufficiency, Theorem~\ref{thm:main} puts
$(1-z)^\rho s_n^\mu$ in $\Sigma_\rho$, and
$(1-z)^{-\rho}$ lies in the same sector. Their product has argument
strictly between $-\rho\pi$ and $\rho\pi$, proving the assertion.

If $\mu>1$, then $s_1^\mu(-1/\mu)=0$ at a point of $\D$.
It remains to treat $0<\rho<1$ and $\nu(\rho)<\mu\leq1$.
Set $b=\rho\pi$, $a=b+\pi$. Since $\sin(t-b)$ changes sign just
once in $(0,a)$, differentiation under the integral gives
\begin{equation}\label{eq:normalized-root-monotonicity}
 \frac{\partial}{\partial\mu}
       \bigl(b^{1-\mu}I(\mu,\rho)\bigr)
 =\int_0^a (t/b)^{\mu-1}\log(t/b)\sin(t-b)\dd t>0.
\end{equation}
The differentiations are justified on compact positive $\mu$-intervals
by an integrable power times $1+|\log t|$.
Thus $I(\mu,\rho)>0$ above its root. On the other hand,
half-wave pairing gives, for $0<\mu\leq1$ and $\pi<a<2\pi$,
\begin{align}\label{eq:unrotated-positive-imaginary}
 \Im F_\mu(a)
 ={}&\int_0^{a-\pi}
 [u^{\mu-1}-(u+\pi)^{\mu-1}]\sin u\dd u\notag\\
 &+\int_{a-\pi}^\pi u^{\mu-1}\sin u\dd u>0.
\end{align}
The first integral is nonnegative and the second strictly positive.
Write $F_\mu(a)=R e^{\ii\varphi}$ with $R>0$ and
$0<\varphi<\pi$. Since
\[
 I(\mu,\rho)=\Im(e^{-\ii\rho\pi}F_\mu(a))
             =R\sin(\varphi-\rho\pi)>0,
\]
and $\varphi-\rho\pi\in(-\pi,\pi)$, we have
$\varphi>\rho\pi$. Lemma~\ref{lem:interior-scaling} then shows
that $\Arg s_n^\mu(z_n)>\rho\pi$ for all sufficiently large
finite $n$, with $z_n\in\D$.

At $\mu=\nu(\rho)$, \eqref{eq:unrotated-positive-imaginary}
and $I(\mu,\rho)=0$ instead give
$F_\mu(a)\in e^{\ii\rho\pi}\mathbb R_{>0}$.
Consequently $\Arg s_n^{\nu(\rho)}(z_n)\to\rho\pi$, proving
angular sharpness for $\rho<1$.

For $\rho=1$, $\nu(1)=1$ by Lemma~\ref{lem:root}, and the
zero above handles every supercritical $\mu$. For critical angular
sharpness use a different sequence, since $F_1(2\pi)=0$.
Put $\zeta_n=e^{2\pi\ii/(n+2)}$ and
$w_n=e^{-1/n^4}\zeta_n\in\D$. The unit-root identity gives
\[
 s_n^1(\zeta_n)=-\zeta_n^{n+1}
                 =e^{\ii n\pi/(n+2)},
 \qquad
 |s_n^1(w_n)-s_n^1(\zeta_n)|
 \leq\frac{n(n+1)}{2n^4}=O(n^{-2}).
\]
The imaginary part of the boundary value is asymptotic to $2\pi/n$,
so for large $n$ the interior values remain in the upper half-plane
and their principal arguments tend to $\pi$.
Complex conjugation supplies the lower boundary in both cases.
\end{proof}

\subsection{The sharp starlike order}
The second narrow sector needed for the function-class result follows
from a standard Herglotz argument, included for its explicit radius
bound.

\begin{lemma}\label{lem:starlike-log-angle}
If $Q\in\mathcal S^*(1-\rho/2)$, $0<\rho\leq1$, then
$Q/z$ is analytic and zero-free. Its logarithm normalized at $0$
satisfies
\begin{equation}\label{eq:starlike-log-angle}
 |\Im\log(Q(z)/z)|\leq\rho\arcsin|z|<\rho\pi/2.
\end{equation}
\end{lemma}
\begin{proof}
The logarithm exists since $Q$ has only its simple zero at the origin.
The Herglotz representation of
$(zQ'/Q-(1-\rho/2))/(\rho/2)$ gives a probability measure
$\sigma$ on the unit circle such that
\[
 \frac{zQ'}Q=1+\rho\int\frac{\xi z}{1-\xi z}\dd\sigma(\xi),
 \qquad
 \log\frac{Q(z)}z=-\rho\int\log(1-\xi z)\dd\sigma(\xi).
\]
The second identity follows by integration with value $0$ at $0$;
all operations are locally uniform. Since
$|\arg(1-\xi z)|\leq\arcsin|z|$, the result follows.
\end{proof}

\begin{theorem}[Optimal order for uniform partial-sum sectors]
\label{thm:starlike-order}
Fix $0<\rho\leq1$ and $0\leq\alpha<1$. The assertion
\begin{equation}\label{eq:starlike-order-quantifiers}
 \frac{S_{n+1}f(z)}z\in\Theta_\rho
 \quad\text{for every }f\in\mathcal S^*(\alpha),\ n\geq0,
 \ z\in\D
\end{equation}
holds if and only if
\begin{equation}\label{eq:optimal-order}
 \alpha\geq1-\frac{\nu(\rho)}2.
\end{equation}
When these equivalent conditions hold, the quotient in
\eqref{eq:starlike-order-quantifiers} has value $1$ at $0$ and is
subordinate to $h_{2\rho}$.
At the critical order the uniform half-angle $\rho\pi$ is optimal.
\end{theorem}
\begin{proof}
Let $M=\nu(\rho)$, $\alpha_0=1-M/2$ and $\beta=1-\rho/2$.
Lemma~\ref{lem:root} gives $\rho\leq M\leq1$, hence
$\alpha_0\leq\beta<1$. For $f\in\mathcal S^*(\alpha_0)$ put
$A=C_M*f$ and $Q=A*K_\rho=\Phi_{\rho,M}*f$.
Then $A\in\mathcal R_{\alpha_0}\subseteq\mathcal R_\beta$ by
\eqref{eq:prestarlike-inclusion}, so the definition of
$\mathcal R_\beta$ gives $Q\in\mathcal S^*(\beta)$.
Lemma~\ref{lem:starlike-log-angle} puts $Q/z$ in $\Sigma_\rho$,
while Theorem~\ref{thm:starlike-convolution} puts
$R_n=S_{n+1}f/Q$ in $\Sigma_\rho$. Taking the two normalized
analytic logarithms shows
\[
 \frac{S_{n+1}f}z=R_n\frac Qz\in\Theta_\rho.
\]
Their argument sum lies strictly in $(-\rho\pi,\rho\pi)$, hence
in $(-\pi,\pi)$; there is no ambiguity of principal argument,
including at $\rho=1$. No convex-hull argument is applied to
$\Theta_\rho$. If $\alpha\geq\alpha_0$, class inclusion gives
the same conclusion. Normalization and the conformality of
$h_{2\rho}$ give the asserted subordination.

For necessity let $\mu=2(1-\alpha)\in(0,2]$. Direct calculation gives
\[
 \frac{zK_\mu'}{K_\mu}
 =\alpha+\frac\mu2\frac{1+z}{1-z},
 \qquad \frac{S_{n+1}K_\mu}z=s_n^\mu.
\]
Thus $K_\mu\in\mathcal S^*(\alpha)$, and
Lemma~\ref{lem:uncorrected-binomial} forces $\mu\leq\nu(\rho)$.
Its critical sequences with $f=K_M$ also prove angular sharpness.
\end{proof}

The case $\rho=1/2$ is the sharp positive-real-part result of
Koumandos--Ruscheweyh \cite{KRG}; its optimal starlike order is
$1-\nu(1/2)/2$. See also \cite[Theorem~4.1]{K2006}.
The case $\rho=1/4$ is \cite[Theorem~4]{KL}, and the
$\rho=1/3$ sufficient statement is \cite[Theorem~5]{SS}.
Theorem~\ref{thm:starlike-order} unifies these conclusions for all
$\rho$ and makes the uniform quantifiers and critical angle explicit.
It concerns the image of $S_{n+1}f/z$, rather than starlikeness of
the section $S_{n+1}f$ itself.

\begin{corollary}[The smallest uniform half-angle]
\label{cor:inverse-angle}
For $1/2\leq\alpha<1$ put
$\rho_\alpha=\nu^{-1}(2(1-\alpha))$. The smallest uniform
half-angle for all $S_{n+1}f/z$, $f\in\mathcal S^*(\alpha)$,
is $\pi\rho_\alpha$. If $0\leq\alpha<1/2$, some such section
has an interior zero, so no sector of the form \eqref{eq:two-sectors}
contains all sections.
\end{corollary}
\begin{proof}
Lemma~\ref{lem:root} gives continuous strict increase of $\nu$ in
$(0,1)$ and the bounds $\rho<\nu(\rho)<\min(1,2\rho)$.
They imply the endpoint limits $0$ and $1$, respectively. Together
with $\nu(1)=1$, this makes $\nu:(0,1]\to(0,1]$ a continuous
strictly increasing bijection. Apply Theorem~\ref{thm:starlike-order}
at $\rho_\alpha$ and its critical angular sharpness. For
$\alpha<1/2$, use $K_{2(1-\alpha)}$ and its two-term section.
\end{proof}

\begin{corollary}[A compact-disc angle bound]\label{cor:compact-angle}
Let $0<\rho\leq1$, $0\leq\alpha<1$ satisfy
\eqref{eq:optimal-order}, and let $f\in\mathcal S^*(\alpha)$.
For every $n\geq0$ and $|z|\leq r<1$,
\begin{equation}\label{eq:compact-angle}
 \left|\Arg\frac{S_{n+1}f(z)}z\right|
 \leq\rho\bigl(\arcsin r+2\arctan r\bigr)<\rho\pi.
\end{equation}
\end{corollary}
\begin{proof}
Use the functions $Q,R_n$ in the sufficient proof above.
Lemma~\ref{lem:starlike-log-angle} bounds the argument of $Q/z$
by $\rho\arcsin r$. Write $R_n=h_\rho\circ\omega_n$.
Schwarz's lemma gives $|\omega_n(z)|\leq r$, and the maximum
absolute argument of $(1+w)/(1-w)$ for $|w|\leq r$ is
$2\arctan r$. Adding the two bounds proves the result.
This is an explicit bound; no optimality at a fixed $r$ is asserted.
\end{proof}

\section{Sharp sectors for Gegenbauer polynomial sections}
\label{sec:gegenbauer}

We use the standard generating-function normalization
\begin{equation}\label{eq:gegenbauer-generating}
 (1-2xz+z^2)^{-\lambda}
 =\sum_{k=0}^\infty C_k^\lambda(x)z^k,
 \qquad
 P_n^\lambda(x,z)=\sum_{k=0}^n C_k^\lambda(x)z^k.
\end{equation}
All assertions below concern the open disc. The previously known
positive-real-part case corresponds to $\rho=1/2$ \cite{KRG};
the $\rho=1/4$ sector result appears in \cite[Corollary~1]{KL},
and the $\rho=1/3$ consequence in \cite[Remark~1]{SS}.

\begin{theorem}[Uniform Gegenbauer sectors]\label{thm:gegenbauer}
Fix $0<\rho\leq1$ and $\lambda>0$. The assertion
\begin{equation}\label{eq:gegenbauer-quantifiers}
 P_n^\lambda(x,z)\in\Theta_\rho
 \quad\text{for every }n\geq0,\ x\in[-1,1],\ z\in\D
\end{equation}
holds if and only if
\begin{equation}\label{eq:gegenbauer-threshold}
 0<\lambda\leq\frac{\nu(\rho)}2.
\end{equation}
At $\lambda=\nu(\rho)/2$ the uniform half-angle $\rho\pi$ is
optimal. The same classification and angular sharpness hold when
the quantifier $x\in[-1,1]$ is replaced by $x\in(-1,1)$.
\end{theorem}
\begin{proof}
Write $x=\cos\theta$ and
\[
 f_x(z)=z(1-e^{\ii\theta}z)^{-\lambda}
              (1-e^{-\ii\theta}z)^{-\lambda}.
\]
Its logarithmic derivative is
\begin{equation}\label{eq:gegenbauer-starlike}
 \frac{zf_x'}{f_x}
 =1-\lambda+\frac\lambda2
 \left(\frac{1+e^{\ii\theta}z}{1-e^{\ii\theta}z}
       +\frac{1+e^{-\ii\theta}z}{1-e^{-\ii\theta}z}\right).
\end{equation}
Each Cayley transform has positive real part in $\D$. Thus, in
the sufficient range $\lambda\leq\nu(\rho)/2\leq1/2$,
$f_x\in\mathcal S^*(1-\lambda)$ and
$1-\lambda\geq1-\nu(\rho)/2$. Since
$S_{n+1}f_x/z=P_n^\lambda(x,\cdot)$,
Theorem~\ref{thm:starlike-order} proves sufficiency, including both
endpoints $x=\pm1$.

At $x=1$ the generating function gives
\[
 C_k^\lambda(1)=\frac{(2\lambda)_k}{k!},\qquad
 P_n^\lambda(1,z)=s_n^{2\lambda}(z).
\]
Lemma~\ref{lem:uncorrected-binomial} proves necessity and critical
angular sharpness. To obtain necessity using only interior $x$, if
$2\lambda>1$ choose $1/(2\lambda)<x<1$; then
$P_1^\lambda(x,-1/(2\lambda x))=0$ in $\D$.
If $\nu(\rho)<2\lambda\leq1$, then necessarily $\rho<1$.
The lemma provides a fixed finite $n$ and $z\in\D$ with
$\Arg P_n^\lambda(1,z)>\rho\pi$ and a nonzero image.
Continuity in $x$ preserves this strict inequality for $x<1$
sufficiently close to $1$.
Likewise, for any proposed smaller critical half-angle, first fix a
finite endpoint example exceeding it, and then move $x$ into
$(-1,1)$. This proves uniform angular sharpness for interior $x$
as well, including $\rho=1$ via the separate sequence in the lemma.
\end{proof}

\begin{corollary}[The exact central specialization]\label{cor:gegenbauer-zero}
For fixed $0<\rho\leq1$ and $\lambda>0$,
\[
 P_n^\lambda(0,z)\in\Theta_\rho
 \quad\text{for every }n\geq0,\ z\in\D
 \quad\Longleftrightarrow\quad 0<\lambda\leq\nu(\rho).
\]
At $\lambda=\nu(\rho)$ the half-angle $\rho\pi$ is optimal.
\end{corollary}
\begin{proof}
At $x=0$, \eqref{eq:gegenbauer-generating} gives the exact identity
\[
 P_n^\lambda(0,z)=s_{\lfloor n/2\rfloor}^\lambda(-z^2).
\]
The map $z\mapsto-z^2$ maps $\D$ onto $\D$, and every index of
a binomial section occurs. Lemma~\ref{lem:uncorrected-binomial}
therefore gives both the classification and its angular sharpness.
\end{proof}

The central specialization doubles the admissible Gegenbauer
parameter relative to the bound uniform in $x$. It is a direct
degeneration of the generating function, not another independent
conjecture. The theorem with all interior $x$ also does not determine
the optimal parameter for each fixed interior $x\ne0$.
Determining that parameter, or a sharp bound uniform on
$|x|\leq b<1$, is a separate question.
The new sections above use no further certified numerical
inequalities: their sufficient statements follow from
Theorem~\ref{thm:main} and classical convolution theory, while
their necessary statements and angular lower bounds follow from
the explicit kernel and interior-scaling arguments.

\appendix

\section{Integral enclosures and the certificate protocol}
\label{app:roots}

For real $\mu\geq0$ and $a>0$, the two regular integrals
\eqref{eq:CS} have the convergent expansions
\begin{equation}\label{eq:CSseries}
\begin{aligned}
 C_0(\mu,a)&=\sum_{j=1}^\infty
        \frac{(-1)^j a^{2j}}{(2j)!(\mu+2j)},\\
 S_0(\mu,a)&=\sum_{j=0}^\infty
        \frac{(-1)^j a^{2j+1}}{(2j+1)!(\mu+2j+1)} .
\end{aligned}
\end{equation}
The acceptance checks retain $j\leq24$. The respective
absolute errors are at most
\begin{equation}\label{eq:CSerrors}
 \frac{a^{50}}{50!(\mu+50)},\qquad
 \frac{a^{51}}{51!(\mu+51)} .
\end{equation}
These follow from the real-variable Taylor remainders for
cosine and sine and integration against $t^{\mu-1}$.
Differentiation in $\mu$ adds a factor $\log t$ to the
integral, so the derivative error bounds are respectively
\[
 \frac{a^{50}}{50!(\mu+50)^2},\qquad
 \frac{a^{51}}{51!(\mu+51)^2}.
\]
All remainder bounds are evaluated with the same complete
parameter interval as the finite sum.

For strictly positive $\mu$, the alternative root check
uses
\begin{equation}\label{eq:rootexp}
\begin{split}
 \frac{\mu I(\mu,\rho)}{a^\mu}
   ={}&-\sin(\rho\pi)
      +\mu\sum_{k=1}^{63}
         \frac{a^k\sin(k\pi/2-\rho\pi)}{k!(\mu+k)}
      +\mathcal E,\\
 &|\mathcal E|\leq\frac{\mu a^{64}}{64!(\mu+64)} .
\end{split}
\end{equation}
The complex exponential has derivatives of absolute value
$1$ on the real axis, giving the last remainder bound.
Thus this check and \eqref{eq:CSseries} enclose the same
integral through different finite expressions.

Here is the acceptance protocol used by the verifiers.
A parameter certificate stores rational intervals $[l,h]$,
endpoint root brackets, and rational normalized angle
intervals $[\alpha,\beta]$ representing
$[\alpha\pi,\beta\pi]$.
The verifier first sorts the parameter intervals and
checks that they have positive lengths, adjacent endpoints
agree exactly, and their union has the stated endpoints.
It performs the same tests for each angle cover.
It checks the strict signs at the two root-bracket endpoints
by \eqref{eq:rootexp}, or by \eqref{eq:F} at the removable
endpoint. Lemma~\ref{lem:root} encloses the interior root
curve; Proposition~\ref{prop:ratio} gives \eqref{eq:qbox}
when a root-ratio enclosure is needed.

The verifier then computes an interval lower bound for the
explicit envelope or prefix on the resulting full rectangle.
Each rational number is converted by enclosing division
of its integer numerator and denominator.
Every arithmetic operation and elementary or special
function is an enclosing ball operation.
The explicit Taylor error ball is added before testing a
strict sign. If an alternative expression loses too much
precision through interval dependence, the verifier bisects
an angle or parameter cell and checks both children.
Any new parameter endpoint receives a new strict root
bracket before use. This terminates on the supplied data
with the cell counts reported in the text.
No stored decimal margin is used as evidence for acceptance.

\section{Uniform bounds at the small-parameter endpoint}
\label{app:small}

This appendix gives a second verification of the entire
endpoint region in Proposition~\ref{prop:certsmall} using
only constant inequalities.
Let
\[
 d_0=10^{-5},\quad q_{\rm lo}=\frac{212}{125},\quad
 q_{\rm hi}=\frac{1697}{1000},\quad
 v=q_{\rm hi}d_0,\quad w=\frac2\pi,\quad N=64,
\]
and put $g_0=e^{-\gamma v}$, $s_0=(q_{\rm hi}-1)d_0$.
Log-convexity of the gamma function gives
$g_0\leq\ga(1+\nu)\leq1$ for $0\leq\nu\leq v$.
The upper bound uses its values at $1$ and $2$; the lower
bound uses the tangent at $1$.
Define
\begin{align*}
 A_0&=g_0(1-1/q_{\rm lo})
       \cos(\pi d_0+s_0\pi/4)\sinc(s_0\pi/4),\\
 V_0&=\pi(1/q_{\rm lo}-1/2)
                     +(1-1/q_{\rm hi})\pi/2,\\
 B_0&=\frac{\pi(1/q_{\rm lo}-1/2)
                        +\pi^{-(1-v)}}{1-v},\\
 f_0&=1+\frac1{6(N+1)},\\
 \mathcal E_0(X,Y)&=
       \frac{1/(4w)+1/(3w^2)}{NX^{1-v}}
                  +\frac{f_0\pi v}{Y^{2-v}} .
\end{align*}
These constants give the right-envelope bound
\begin{equation}\label{eq:smallconstantR}
\begin{split}
 G_R\geq{}&
 A_0-\frac{V_0\pi}{12w^v}
       -\frac{B_0d_0^2\pi}{4w}
       -\frac{(1+d_0)d_0}{\pi^{2-v}w}
       -\mathcal E_0(2\pi/3,2\pi/3)\\
 &>\frac18 .
\end{split}
\end{equation}
To see this, the first and fourth terms in
\eqref{eq:GRstable} are bounded together by
$(A_0-V_0\pi/(12w^v))/u$. The numerator is positive,
so this is at least the same numerator.
Also $b_R\geq\pi$, $p\geq1-v$, whence
$|\widetilde M_R|\leq B_0$. Bound the possible negative
second term by $B_0d_0^2\pi/(4w)$.
The positive-tail loss is at most the fourth term on the
right of \eqref{eq:smallconstantR}; its factor $1+d_0$
is a harmless enlargement. Finally
$X_R,Y_R\geq2\pi/3>1$, so $\mathcal E_0$ bounds both
discrete losses.
The numerical lower endpoint of the full constant
expression is greater than $0.1586$.

For the left side put
\[
 b_0=(1-d_0)\frac{3\pi}{2},\qquad
 m_0=g_0\frac{\pi}{2q_{\rm hi}}
                  \sinc(v\pi/2)w-b_0^{-(1-v)} .
\]
One has $b_L\geq b_0$,
$V_L\geq\pi/(2q_{\rm hi})$,
$\sinc(\nu V_L)\geq\sinc(v\pi/2)$, and $u^p\geq u\geq w$.
Thus the numerator in \eqref{eq:MLstable} is at least $m_0$.
The constant $m_0$ is positive; as $p\leq1$, $M_L\geq m_0$.
Consequently
\begin{equation}\label{eq:smallconstantL}
 G_L\geq m_0-
       \mathcal E_0\left(\frac{2N\pi}{N+1},2\pi\right)
                    >\frac13 .
\end{equation}
The verified lower endpoint is greater than $0.3739$.
The function \file{small_endpoint_analytic} in
\file{check_extension.py} evaluates
\eqref{eq:smallconstantR} and \eqref{eq:smallconstantL}
with enclosing ball arithmetic, and separately rechecks
the root-ratio signs in \eqref{eq:smallq}.

\end{document}